\documentclass[reqno,11pt,a4]{amsart}
 
\usepackage{a4wide}
\usepackage{algorithmicx}
\usepackage[latin1]{inputenc}
\usepackage{amsmath,amsthm,amssymb,bbm,mathrsfs}
\usepackage{color,epsfig,epstopdf}
\usepackage{ucs}
\usepackage{units}
\usepackage{stmaryrd}
\usepackage{enumitem}
\usepackage{comment}
\usepackage{hyperref}
\usepackage{tikz}
\usepackage[english,norelsize,algoruled,lined,noresetcount,linesnumbered]{algorithm2e}

\usepackage{epsfig,psfrag,color,ifthen}

\definecolor{darkgreen}{rgb}{0.00,0.50,0.10}
\definecolor{lightgreen}{rgb}{0.20,0.70,0.30}

\newtheorem{theorem}                   {Theorem}%[section]

\newtheorem{lemma}           [theorem] {Lemma}

\newtheorem{claim}            {Claim} 
\newtheorem{fact}			 [theorem]{Fact}		
 
\newtheorem{definition}      [theorem] {Definition} 
\newtheorem{example}         [theorem] {Example} 
\newtheorem{conjecture}      [theorem] {Conjecture}

\newtheorem{remark}          [theorem] {Remark}
\newtheorem{setup}           [theorem] {Setup}

\newcommand{\reminder}[1]{\marginpar{\textcolor{red}{#1}}}

\usepackage{chngcntr}
\counterwithin{figure}{section}

\let\subset\subseteq
\let\epsilon\varepsilon
\let\eps\varepsilon
\let\rho\varrho
\def\Pc{\eta} %This should latex be changed to probably \eta

\setitemize{leftmargin=\parindent}
\setenumerate{leftmargin=*}

\def\dbigcup{\dot\bigcup}
\def\maxdeg{\Delta}
\newcommand{\JUSTIFY}[1]{\fbox{\tiny{#1}}\quad}

\newcommand{\By}[2]{\overset{\mbox{\tiny{#1}}}{#2}}
\newcommand{\ByRef}[2]{   \By{\eqref{#1}}{#2} }
\newcommand{\eqBy}[1]{    \By{#1}{=} }

\newcommand{\leBy}[1]{    \By{#1}{\le} }

\newcommand{\eqByRef}[1]{ \ByRef{#1}{=} }

\newcommand{\leByRef}[1]{ \ByRef{#1}{\le} }
\newcommand{\geByRef}[1]{ \ByRef{#1}{\ge} }

\newcommand{\Exp}{\mathbf{E}} %expectation
\newcommand{\Prob}{\mathbf{P}} %probability
\newcommand{\bbN}{\mathbb{N}} %natural numbers
\newcommand{\bbZ}{\mathbb{Z}} %integers
\newcommand{\bbA}{\mathbb{A}} %available vertex labels
\newcommand{\bbC}{\mathbb{B}} %available edge labels
\newcommand{\setVert}{\mathsf{A}}
\newcommand{\setEdge}{\mathsf{B}}
\newcommand{\cI}{\mathcal{I}_{\mathsf V}}
\newcommand{\iI}{I}
\newcommand{\cIe}{\mathcal{I}_{\mathsf E}}
\newcommand{\iIe}{I_{\mathsf E}}
\newcommand{\cJ}{\mathcal{J}}
\newcommand{\iJ}{J}
\newcommand{\cX}{\mathcal{X}}
\newcommand{\Xone}[1]{X_1\ifthenelse{\equal{#1}{none}}{}{\left\llbracket #1 \right\rrbracket}} % the ifthenelse allows for \Xone{none} to just say the name without brackets
\newcommand{\Xtwo}[1]{X_2\ifthenelse{\equal{#1}{none}}{}{\left\llbracket #1 \right\rrbracket}} %used to be X_3
\newcommand{\Xthree}[1]{X_3\ifthenelse{\equal{#1}{none}}{}{\left\llbracket #1 \right\rrbracket}} %used to be X_4
\newcommand{\Xfour}[1]{X_4\ifthenelse{\equal{#1}{none}}{}{\left\llbracket #1 \right\rrbracket}} %used to be X_5
\newcommand{\ind}{\mathbbm{1}}
\newcommand{\chosenv}[2]{\mathrm{Ch_{\mathsf{V}}}(#1;#2)}
\newcommand{\chosene}[2]{\mathrm{Ch_{\mathsf{E}}}(#1;#2)}
\newcommand{\el}[1]{\mathfrak{Q}(#1)}
\newcommand{\ed}{\mathfrak{e}}
\newcommand{\cF}{\mathcal{F}}
\newcommand{\cE}{\mathcal{E}}
\newcommand{\hist}{\mathscr{H}}
\newcommand{\Vh}{\mathcal{D}_h}

\newcommand{\Corv}{\mathrm{Cor_{\mathsf V}}} %vertex correction
\newcommand{\Core}{\mathrm{Cor_{\mathsf E}}} %edge corrections
\newcommand{\diff}{\mathrm{Diff}}

\newcommand{\freev}{\mathrm{Free_{\mathsf V}}}
\newcommand{\freee}{\mathrm{Free_{\mathsf E}}}
\newcommand{\free}{\mathrm{free}}

\newcommand{\adm}{\mathrm{Adm}} %the admissible vertices

\newcommand{\prt}[1]{{#1^+}}
\PrerenderUnicode{\unichar{355}}

\newcommand{\oldqed}{}
\def\endofClaim{\hfill\scalebox{.6}{$\Box$}}

\newenvironment{claimproof}[1][Proof]{
  \renewcommand{\oldqed}{\qedsymbol}
  \renewcommand{\qedsymbol}{\endofClaim}
  \begin{proof}[#1]
}{
  \end{proof}
  \renewcommand{\qedsymbol}{\oldqed}
}

\begin{document}

\title{Almost all trees are almost graceful}
\author[A. Adamaszek]{Anna Adamaszek}\address{\emph{Affiliation at time of submission:} Department of Computer Science, University of Copenhagen, Universitetsparken~5, 2100 Copenhagen, Denmark.}
\email{a.m.adamaszek@gmail.com}
\author[P. Allen]{Peter Allen}
\address{Department of Mathematics, London School of Economics, Houghton Street, London, WC2A~2AE, UK.}
\email{p.d.allen@lse.ac.uk}
\author[C. Grosu]{Codru\unichar{355} Grosu}
\address{\emph{Affiliation at time of submission:} Freie Universit\"at Berlin, Germany.} 
\email{grosu.codrut@gmail.com}
\author[J. Hladk\'y]{Jan Hladk\'y}
\address{Institute of Mathematics of the Czech Academy of Sciences, \v Zitn\'a 25, Praha.} \email{hladky@math.cas.cz}
\thanks{{\it AA:} Most of the work was done while the author was at Max-Planck-Institut f\"ur Informatik, Saarbr\"ucken, Germany supported by the Humboldt Foundation. 
{\it CG:} This research was supported by the Deutsche Forschungsgemeinschaft within the research training group `Methods for Discrete Structures' (GRK 1408). {\it JH:} The research leading to these results has received funding from the People Programme (Marie Curie Actions) of the European Union's Seventh Framework Programme (FP7/2007-2013) under REA grant agreement number 628974.
The Institute of Mathematics is supported by RVO:67985840}

\begin{abstract}
The Graceful Tree Conjecture of Rosa from 1967 asserts that the vertices of each tree $T$ of order $n$ can be injectively labelled by using the numbers $\{1,2,\ldots,n\}$ in such a way that the absolute differences induced on the edges are pairwise distinct.

We prove the following relaxation of the conjecture for each $\gamma>0$ and for all $n>n_0(\gamma)$. Suppose that \emph{(i)} the maximum degree of $T$ is bounded by $O_\gamma(n/\log n$), and \emph{(ii)} the vertex labels are chosen from the set $\{1,2,\ldots,\lceil (1+\gamma)n\rceil\}$. Then there is an injective labelling of $V(T)$ such that the absolute differences on the edges are pairwise distinct. In particular, asymptotically almost all trees on $n$ vertices admit such a labelling.

As a consequence, for any such tree $T$ we can pack $\lceil(2+2\gamma)n\rceil-1$ copies of $T$ into $K_{\lceil(2+2\gamma)n\rceil-1}$ cyclically. This proves an approximate version of the Ringel--Kotzig conjecture (which asserts the existence of a cyclic packing of $2n-1$ copies of any $T$ into $K_{2n-1}$) for these trees.

The proof proceeds by showing that a certain very natural randomized algorithm produces a desired labelling with high probability.
\end{abstract}
\maketitle

%%%% Introduction %%%%%%%%%%%%%%%%%%%%%%%%%%%%%%%%%%%%%%%%%%%%%%%

\section{Introduction}

\subsection{Graceful labelling}

Let $G$ be a graph with $n$ vertices and $q$ edges.
A \emph{vertex labelling} of $G$ is an assignment of natural numbers to the vertices of $G$, subject to certain conditions. A vertex labelling $f$ is called \emph{graceful} if
$f$ is an injection from $V(G)$ to the set $\{1,\ldots,q+1\}$ such that, if each edge $xy \in E(G)$ is assigned the \emph{induced} label $|f(x)-f(y)|$, then the resulting edge labels are distinct. The graph $G$ is called \emph{graceful} if it admits a graceful labelling.

Graceful labellings were first introduced by Rosa~\cite{Rosa66} under the name of \emph{$\beta$-valuations}. 
It was Golomb \cite{Golomb72} who used the term \emph{graceful} for the first time.

A natural problem associated with graceful labellings is to determine which graphs are graceful. According to an unpublished result of Erd\H os almost all graphs are not graceful. A version of this argument for a very similar concept of the so-called harmonious labellings, which we introduce in detail in Section~\ref{ssec:harmonious}, was later recorded by Graham and Sloane \cite{GrahamSloane80}. It was shown by Rosa \cite{Rosa66} that a graph with every vertex of even degree and number of edges congruent to $1$ or $2$ (mod $4$) is not graceful. Nevertheless, it appears that many graphs that exhibit some regularity in structure are graceful. For example, paths $P_n$ and wheels $W_n$ are graceful, \cite{Rosa66}, \cite{Frucht79}. A comprehensive survey on the current status of knowledge on graceful graphs can be found in \cite{Gallian09}.

The most important problem in the area is to determine whether every tree is graceful. Conjectured first by Rosa in $1967$, the problem remains wide open.

\begin{conjecture}[Graceful Tree Conjecture]
\label{conj:GTLC}
For any $n$-vertex tree $T$ there exists an injective
labelling $\psi:V(T) \rightarrow \{1,\ldots,n\}$ that induces pairwise
distinct labels on the edges of~$T$.
\end{conjecture}

While Conjecture~\ref{conj:GTLC} has attracted a lot of attention, it was proved only for some special classes of trees (paths and caterpillars \cite{Rosa66}, firecrackers \cite{Chen97}, banana trees \cite{Sethuraman09}, olive trees \cite{Pastel78}, trees of diameter at most $7$ \cite{Wang15}, and several others, see Table $4$ in \cite{Gallian09}). The related conjecture of Bermond \cite{Bermond79} that all lobsters are graceful is still open, and has also been the subject of many papers.

Van Bussel~\cite{Bussel02} introduced the following relaxation of gracefulness. A map $\psi:V(G)\rightarrow[m]$ from a vertex set of a graph is \emph{$m$-graceful} if $\psi$ is injective, and the map $\psi_*$ induced on the edges, $\psi_*:E(G)\rightarrow[m-1]$, $\psi_*(xy):=|\psi(x)-\psi(y)|$, is injective as well. Thus, to obtain the original (nonparametric) definition, we take $m=|E(G)|+1$.  If the codomain $[m]$ of an $m$-graceful map is clear from the context, we simply call $\psi$ \textit{graceful}. Van Bussel proved that every tree $T$ admits a $\big(2v(T)-2\big)$-graceful labelling.

\subsection{Tree packings}\label{ssec:treepackings}
The motivation for considering graceful labellings comes from the area of graph packings. A collection $G_1,\dots,G_t$ of graphs \emph{pack} into a graph $H$ if there are embeddings $\psi_1,\dots,\psi_t$ of $G_1\dots,G_t$ into $H$ such that each edge of $H$ is used in at most one embedding. If in addition each edge of $H$ is used in some embedding, we say that $G_1,\dots,G_t$ \emph{decompose} $H$.

There are several open problems in the area of graph packings. These are special cases of the following general meta-conjecture: if $G_1,\dots,G_t$ are drawn from a family of `sparse' graphs, and $H$ is `dense', then provided some `simple' necessary conditions are satisfied, $G_1,\dots,G_t$ pack into $H$.

In particular, we can consider the (sparse) family of trees and the (dense) complete graph. Even here there are several incarnations of the meta-conjecture. The one which will concern us is Ringel's conjecture from 1963.

\begin{conjecture}[Ringel's conjecture, \cite{Ringel63}]
\label{conj:ringel}
For any $(n+1)$-vertex tree $T$ the complete graph $K_{2n+1}$ can be decomposed into $2n+1$ edge-disjoint subgraphs isomorphic to $T$.
\end{conjecture}

This conjecture can be strengthened by requiring the embeddings to have a special structure. Specifically, suppose that the vertices of $K_{2n+1}$ are the integers $0,1,\ldots,2n$. For any subgraph $G$ of $K_{2n+1}$ we may define the \emph{cyclic shift} of $G$ as the subgraph $S(G)$ with 
\begin{equation*}
\label{eq:cyclic_shift}
S(G)= \left(\{x+1 \::\: x \in V(G)\}, \{(x+1,y+1) \::\: (x,y) \in E(G)\}\right)
\end{equation*}
where all addition is performed modulo $2n+1$.

If $G$ is any graph with $n$ edges, we say that $K_{2n+1}$ can be \emph{cyclically decomposed} into copies of $G$ if there is a subgraph $G' \simeq G$ of $K_{2n+1}$ such that the cyclic shifts $G',S(G'),\ldots,S^{2n}(G')$ are edge-disjoint (and thus form a decomposition of $K_{2n+1})$.

As reported by Rosa \cite{Rosa66}, the following conjecture is due to Kotzig.
\begin{conjecture}[Ringel--Kotzig conjecture]
\label{conj:kotzig}
For any $(n+1)$-vertex tree $T$ the complete graph $K_{2n+1}$ can be cyclically decomposed into copies of $T$.
\end{conjecture}

Finally, we can connect this to graceful labellings. If $T$ has a graceful labelling, then $T$ satisfies the Ringel--Kotzig conjecture. Furthermore, if $T$ has an $m$-graceful labelling, then we can cyclically pack $2m-1$ copies of $T$ into $K_{2m-1}$. Specifically, let $\psi:V(T)\to [m]$ be an $m$-graceful labelling of $T$. Then $\psi$ is also an embedding of $T$ into $K_{2m-1}$, since $\psi$ is injective. We claim the cyclic shifts of $\psi$ form a packing of $2m-1$ copies of $T$ into $K_{2m-1}$. Indeed, by symmetry we only have to check that if $uv$ is an edge of $T$, then $\psi(u)\psi(v)$ is not the image of any $u'v'\in E(T)$ under any non-trivial power of the cyclic shift $\psi'$ of $\psi$. Without loss of generality we may assume $1\le \psi(u)<\psi(v)\le m$. If the range of $\psi'$ contains both $\psi(u)$ and $\psi(v)$ then it must contain the interval from $\psi(u)$ to $\psi(v)$. If $u'v'\in E(T)$ satisfies $\psi'(u')\psi'(v')=\psi(u)\psi(v)$, then we have
\[\big|\psi(u)-\psi(v)\big|=\big|\psi'(u')-\psi'(v')\big|=\big|\psi(u')-\psi(v')\big|\;,\]
where the final step uses that the range of $\psi'$ contains the interval from $\psi(u)$ to $\psi(v)$. Now since $\psi$ is graceful, we have $u'v'=uv$, and hence $\psi'=\psi$. 

Conjecture~\ref{conj:ringel} was wide open until recently when in \cite{BottHlaPigTar:TreePack} an approximate version for bounded degree trees was proven.\footnote{The main result in~\cite{BottHlaPigTar:TreePack} is more general and allows almost perfect decompositions of a complete graph with an arbitrary family of almost-spanning bounded degree trees, addressing also the Tree Packing Conjecture of Gyarf\'as.} The main result of \cite{BottHlaPigTar:TreePack} was extended by Messuti, R\"odl, and Schacht~\cite{MeRoSch:Packing} to permit almost-spanning bounded degree graphs with limited expansion and later Ferber, Lee, and Mousset~\cite{FeLeMo:Packing} allowed for spanning bounded degree graphs with limited expansion. Kim, K\"uhn, Osthus, and Tyomkyn~\cite{KiKuOsTy:Packing} were able to obtain almost perfect decompositions of a complete graph with an arbitrary family of bounded degree graphs, and using this Joos, Kim, K\"uhn and Osthus~\cite{JoKiKuOs:Packing} proved the Tree Packing Conjecture and Ringel's Conjecture exactly for arbitrary bounded degree trees. Ferber and Samotij~\cite{FeSa:PackingTrees} addressed the problem of packing trees with unbounded degrees, in particular packing almost-spanning trees with maximum degree $cn/\log n$ into complete graphs. However the methods used in these papers do not seem not to hint at any approaches for graceful tree labellings. After this paper was made public, Allen, B\"ottcher, Hladk\'y, and Piguet~\cite{ABHP:DegeneratePacking} showed an almost perfect decomposition result for any family of possibly spanning graphs of with maximum degree $cn/\log n$ and constant degeneracy, thus improving (in the setting of complete host graphs) upon~\cite{BottHlaPigTar:TreePack,MeRoSch:Packing,FeLeMo:Packing,KiKuOsTy:Packing}. Some ideas used in~\cite{ABHP:DegeneratePacking} are inspired by the present work. The strongest result on Conjecture~\ref{conj:ringel} for general trees (with no degree restriction) is a result of Montgomery, Pokrovskiy, and Sudakov~\cite{MPS:EmbeddingRainbow}, who proved that $2\ell+1$ copies of any $(n+1)$-vertex tree $T$ pack into $K_{2\ell+1}$ whenever $\ell\ge n+o(n)$.\footnote{Theorem~1.3 in~\cite{MPS:EmbeddingRainbow} contains a slightly weaker statement, but this what its proof actually gives.} Actually, such a packing follows immediately by a similar reduction as the one given after Conjecture~\ref{conj:kotzig} and their main result that such a tree admits a harmonious labelling by a group $\bbZ_{n+o(n)}$.

\subsection{Our result}
In this paper we prove an approximate version of Conjecture~\ref{conj:GTLC} for trees with maximum degree $o\big(\tfrac{n}{\log n}\big)$. This implies approximate versions of the Ringel--Kotzig and Ringel conjectures.

\begin{theorem}
\label{thm:main}
For every $\gamma>0$ there exist $\Pc>0$ and $n_0\in\mathbb{N}$ such that the
following holds for every $n>n_0$. Suppose that $T$ is an $n$-vertex tree and
$\maxdeg(T)\le \tfrac{\Pc n}{\log n}$. Then there exists a
$\lceil(1+\gamma)n\rceil$-graceful labelling $\psi:V(T)\rightarrow [\lceil(1+\gamma)n\rceil]$.
\end{theorem}

This theorem applies to random trees. More precisely, we work with the set $\mathcal T_n$ of all labelled trees on the vertex set $\{1,\ldots,n\}$. By a classical result of Moon~\cite{Moon68}, a tree selected uniformly at random from the set $\mathcal T_n$ has maximum degree $o(\log n)$ with probability tending to~$1$ as $n$ tends to infinity. In particular Theorem~\ref{thm:main} applies to almost all trees.

\medskip
Our proof of Theorem~\ref{thm:main} is an application of the Probabilistic Method, inspired by~\cite{BottHlaPigTar:TreePack}. More precisely, the proof is an application of the Differential Equations Method (DEM). That is, we run a certain randomized algorithm which sequentially labels the vertices of the $n$-vertex tree $T$, and we prove that with high probability this process produces a $[\lceil(1+\gamma)n\rceil]$-graceful labelling of $T$. As the algorithm progresses with the labelling, the sets of (edge- and vertex-) labels available for future steps keep getting sparser. The key for the analysis of the algorithm is to introduce suitable measures of quasirandomness for these sets, and prove that the sets of available labels evolve in a quasirandom way.

Let us note that DEM has been used extensively in discrete mathematics in particular thanks to the tools developed by Wormald~\cite{Wormald}. We do not use Wormald's machinery and it is not clear to us whether that formalism applies in our setting at all. To get a handle on various parameters during the run of the process --- this handle provided in other scenarios by DEM itself --- we introduce in Section~\ref{sec:notation} a variant of Hoeffding's bound. 

After introducing notation and some preliminary facts in Section~\ref{sec:notation}, we outline the proof of Theorem~\ref{thm:main} in Section~\ref{sec:setup}. In Section~\ref{sec:setup}, we also give a detailed description of our labelling algorithm and introduce the key quasirandomness concepts. The missing bits in our proof that the algorithm will produce a labelling required in Theorem~\ref{thm:main} are given in Section~\ref{sec:proof}. In Section~\ref{sec:conclusion} we suggest various strengthenings of our main result.
%%%% Notation %%%%%%%%%%%%%%%%%%%%%%%%%%%%%%%%%%%%%%%%%%%%%%%

\section{Notation and auxiliary results}
\label{sec:notation}
For a graph $G$, the \textit{order} of $G$ is the number of vertices of $G$. We write $\Delta(G)$ for the maximum degree of $G$.

We write $a=b \pm \epsilon$ when we have $a\in[b-\eps,b+\eps]$. Extending this, and in a slight abuse of notation, we write $a \pm \delta = b \pm \eps$ for the inclusion $[a - \delta, a + \delta] \subseteq [b - \eps, b+\eps]$. We write $\log$ for the natural logarithm.

We use $\Prob[\cdot]$ and $\Exp[\cdot]$ to denote the probability and the expectation, respectively. All probability spaces considered in this paper are finite. In such a setting, any sigma-algebra is generated by its inclusion-wise minimal nonempty sets, which naturally form a partition of the probability space, and it is convenient to work with the partition rather than the sigma-algebra it generates.

Recall that if $\Omega$ is a finite probability space then a sequence of partitions $(\mathcal{F}_0$, $\mathcal{F}_1$,\dots,  $\mathcal{F}_n)$ of $\Omega$ is a \emph{filtration} if for each $i\in[n]$, the partition $\mathcal{F}_i$ refines $\mathcal{F}_{i-1}$. Recall that given a function $f:\Omega\rightarrow \mathbb{R}$, its \emph{conditional expectation with respect to $\mathcal{F}_i$}, denoted by $\Exp[f|\mathcal{F}_i]$, is a function $\Exp[f|\mathcal{F}_i]:\Omega\rightarrow\mathbb{R}$ defined by $\Exp[f|\mathcal{F}_i](\omega)=\Exp[f|X]$, where $X\in\mathcal{F}_i$ is the cell containing $\omega$.
Recall also that in this setting, a function $f:\Omega\rightarrow \mathbb{R}$ is \emph{$\mathcal{F}_i$-measurable} if $f$ is constant on each cell of $\mathcal{F}_i$. In this paper, $\Omega$ will be the probability space on which a probabilistic process is defined, the cells of $\mathcal{F}_i$ will be given by all the random choices made up to some time $t_i$ in the process, so choosing $t_i$ increasing in $i$ automatically gives a filtration. In this setting, if a function $f$ is $\mathcal{F}_i$-measurable, that means that its value is fixed by the random choices made in the process up to time $t_i$.

%In such a setting, we always work with the full sigma-algebra generated by all singletons. 

\subsection{Hoeffding's bound}The following theorem of Hoeffding~\cite[Theorem~2]{Hoeff} gives bounds on sums of independent real-valued bounded random variables.

\begin{theorem}\label{thm:hoeff}
 Let $Y_1,\ldots,Y_n$ be independent random variables with $0\le Y_i\le a_i$ for each $i\in[n]$. Let $X=Y_1+\dots+Y_n$, and let $\mu=\Exp X$. Then we have
 \begin{align*}
  \Prob[X-\mu\ge t]&\le \exp\left(-\frac{2t^2}{\sum_{i=1}^na_i^2}\right)\quad\text{and}\\
  \Prob[X-\mu\le -t]&\le \exp\left(-\frac{2t^2}{\sum_{i=1}^na_i^2}\right)\,.
 \end{align*}
\end{theorem}
Theorem~\ref{thm:hoeff} is one of the most commonly used concentration bounds in probabilistic combinatorics. We shall need an extension of this theorem to non-independent random variables. 

\begin{lemma}\label{lem:seqhoeff}
 Let $\Omega$ be a finite probability space, and $(\cF_0,\dots,\cF_n)$ a filtration. Suppose that for each $1\le i\le n$ we have a nonnegative real number $a_i$, an $\cF_i$-measurable random variable $Y_i$ satisfying $0\le Y_i\le a_i$, nonnegative real numbers $\mu$ and $\nu$, and an event $\cE$ with $\Prob[\cE]>0$. Suppose that almost surely, either $\cE$ does not occur or $\sum_{i=1}^n\Exp\big[Y_i\big|\cF_{i-1}\big]=\mu\pm\nu$. Then for each $t>0$ we have
 \[\Prob\Big[\cE\text{ and }\Big|\sum_{i=1}^nY_i-\mu\Big|\ge\nu+t\Big]\le 2\exp\Big(-\frac{2t^2}{\sum_{i=1}^na_i^2}\Big)\,.\]
\end{lemma}
In applications, $\cE$ will be an event that holds with probability close to~1. Then, Lemma~\ref{lem:seqhoeff} controls large deviations of $\sum_{i=1}^nY_i$.
\begin{proof}
 Given $Y_1,\dots,Y_n$, we define random variables $U_1,\ldots,U_n$ by $U_i=Y_i$ if $\Prob[\cE|\cF_{i-1}]>0$, and $U_i=0$ otherwise. Note that $U_i$ is $\cF_i$-measurable for each $i$. Furthermore, we claim that for each $1\le k\le n$ we have almost surely 
 \begin{equation}\label{eq:happy}
 \sum_{i=1}^k\Exp[U_i|\cF_{i-1}]\le\mu+\nu\;.
 \end{equation}
Indeed, suppose $k$ is minimal such that this statement fails, and let $F$ be a set in $\cF_{k-1}$ with $\Prob[F]>0$ witnessing its failure. By minimality of $k$, we have $U_k>0$, so $\Prob[\cE|F]>0$. Thus with positive probability, $\cE$ occurs and $\sum_{i=1}^n\Exp[U_i|\cF_{i-1}]\ge\sum_{i=1}^k\Exp[U_i|\cF_{i-1}]>\mu+\nu$, contradicting the assumption of the lemma.
 
Set $Z_0:=0$ and for $i=1,\ldots,n$ set $Z_i:=\sum_{j=1}^i(U_j-\Exp[U_j|\cF_{j-1}])$. It is straightforward to check that $(Z_i)_{i=0}^n$ is a martingale which satisfies $|Z_i-Z_{i-1}|=|U_i-\Exp[U_i|\cF_{i-1}]|\le a_i$. Thus, Azuma's Inequality (see e.g.~\cite[Theorem 2.25]{JLR:RandomGraphs}) gives us
\begin{align*}
\exp\Big(-\frac{2t^2}{\sum_{i=1}^na_i^2}\Big)&\ge
\Prob\Big[Z_n-Z_0\ge t\Big]
=\Prob\left[ \sum_{i=1}^n (U_i -\Exp[U_i|\cF_{i-1}]) \ge t\right]\\
\JUSTIFY{by~\eqref{eq:happy}}&\ge
\Prob\left[ \sum_{i=1}^n U_i \ge \mu+\nu+t\right]\;.
\end{align*}
If $\cE$ occurs then almost surely $Y_i=U_i$ for each $1\le i\le n$. Therefore, we have
 \[\Prob\Big[\cE\text{ and }\sum_{i=1}^nY_i\ge\mu+\nu+t\Big]\le\exp\Big(-\frac{2t^2}{\sum_{i=1}^na_i^2}\Big)\,.\]
 
 The same argument applied to the random variables $a_i-Y_i$ gives
 \[\Prob\Big[\cE\text{ and }\sum_{i=1}^nY_i\le\mu-\nu-t\Big]\le\exp\Big(-\frac{2t^2}{\sum_{i=1}^na_i^2}\Big)\,,\]
 and so the lemma statement holds by the union bound.
\end{proof}

\subsection{Cutting a tree}
The following lemma (variants of which are well-known) tells us that trees can be easily separated into small components.
\begin{lemma}\label{lem:cuttree}
 For any $\eps\in(0,1)$, any $n$ such that $\eps n\ge 2\log n$, and any tree $T$ with $\Delta(T)\le\tfrac{\eps^2n}{4\log n}$ and $v(T)\le n$, there exists a set $\mathcal{R}$ of edges of $T$ with $|\mathcal{R}|\le\eps v(T)$ such that the components of $T-\mathcal{R}$ have order at most $\tfrac{\eps n}{\log n}$.
\end{lemma}
\begin{proof}
 We prove the lemma by induction on $v(T)$. The statement is trivially true for $v(T)\le\tfrac{\eps n}{\log n}$, so we may assume $v(T)>\tfrac{\eps n}{\log n}$. It is enough to show that there exists one edge $uv\in E(T)$ such that one of the two components of $T-uv$ has order between $2\eps^{-1}$ and $\tfrac{\eps n}{\log n}$, since then the statement follows by applying the induction hypothesis to the other component.
 
 We find the edge $uv$ by the following `walk' procedure. We start with a leaf vertex $u_1$ and its neighbour $v_1$. Now for each $t\ge 1$ in succession we do the following. If the component of $T-u_tv_t$ containing $v_t$ has order between $2\eps^{-1}$ and $\tfrac{\eps n}{\log n}$, we stop and return $uv=u_tv_t$. If not, we set $u_{t+1}=v_t$ and $v_{t+1}$ to be the neighbour of $u_{t+1}$ not equal to $u_t$ which maximises the order of the component of $T-u_{t+1}v_{t+1}$ containing $v_{t+1}$.
 
 If at time $t$ the component of $T-u_tv_t$ containing $v_t$ has more than $\tfrac{\eps n}{\log n}$ vertices, then by averaging, the component of $T-u_{t+1}v_{t+1}$ containing $v_{t+1}$ has at least
 \[\tfrac{1}{\Delta(T)}\big(\tfrac{\eps n}{\log n}-1\big)\ge\tfrac{1}{\Delta(T)}\cdot\tfrac{\eps n}{2\log n}\ge 2\eps^{-1}\] vertices, where the first inequality is by choice of $n$ and the second by the bound on $\Delta(T)$. Thus the algorithm finds the desired $uv$.
\end{proof}

%%%% Proof of the main theorem %%%%%%%%%%%%%%%%%%%%%%%%%%%%%%%%%%

\section{Setup}
\label{sec:setup}

Before we embark on the proof of Theorem~\ref{thm:main}, we outline the main ideas. We write $\tilde n=\lceil(1+\gamma)n\rceil$. We wish to label our tree in a random process, at each step labelling one new vertex which has one labelled neighbour (so that the set of labelled vertices is always a subtree), and choosing the new vertex label to be admissible, that is, to avoid re-using either the vertex label or the induced edge label. We keep track of the sets of vertex and induced edge labels which remain available at the $t$\textsuperscript{th} step, for which we will write $\setVert_t$ and $\setEdge_t$ respectively. We will show that $\setVert_t$ looks like a uniform random subset of $[\tilde n]$ of cardinality $\tilde n-t$ and that $\setEdge_t$ looks like a uniform random subset of $[\tilde n-1]$ of cardinality $\tilde n-t-1$. The reason why we do this is that if $\setVert_t$ and $\setEdge_t$ were really uniformly distributed it would be easy to show that at every step there is (with probability very close to~1) an admissible choice for the new label.

Choosing the new labels uniformly from all admissible labels will not lead to random-looking sets $\setVert_t$ and $\setEdge_t$. Let us illustrate this in the initial situation when, starting with $\setVert_0=[\tilde n]$ and $\setEdge_0=[\tilde n-1]$, we label first two neighboring vertices $v_1$ and $v_2$. Then the first created edge label on $v_1v_2$ will be assigned the smallest possible value of~$1$  with probability $\frac{2}{\tilde n}$, but the largest possible value of~$\tilde n-1$  with probability only $\frac{2}{\tilde n(\tilde n-1)}$, and the probabilities of labels with intermediate values interpolate these two extremes. This effect will persist, so that $\setEdge_t$ will rapidly become non-uniform, being much sparser for small labels than large. That will in turn cause $\setVert_t$ to evolve in a non-uniform way as well.

But observe that if $uv$ is an edge of $T$, if we choose $\psi(u)$ a uniformly random vertex label, and then choose $\psi(v)$ randomly in a small interval around $\tilde n-\psi(u)$ then the distributions of each of $\psi(u)$, $\psi(v)$ and the induced edge label $\big|\psi(u)-\psi(v)\big|$ are close to the uniform distribution. Thus we would like to label our tree such that most edges and vertices are labelled in about this way.

We will label the vertices of $T$ in the order $v_1,\dots,v_n$ (which we will determine), and in each case we choose an admissible label in a certain short interval $\iJ(v_i)$ within $[\tilde n]$. We choose the $\iJ(v_i)$ to be uniformly distributed\footnote{\label{foot:technicalcomplication}A technical complication arises in that the extremes of the interval $[\tilde n]$ are not covered as much as the rest; we will ignore this for now.} over $[\tilde n]$, and such that for most $v_i$ the intervals $\iJ(v_i)$ and $\iJ(v_j)$, where $v_j$ is the parent of $v_i$ in the ordering, are equally far from and on opposite sides of $\tfrac12(1+\gamma)n$.

In reality, we certainly do not choose vertex labels uniformly at random in the intervals; there is a great deal of dependency in the label choices. However, that dependency is sequential: conditioning on the labelling history of the first $t-1$ vertices, we know the distribution of the $t$\textsuperscript{th} vertex, and we will be able to show that the average of these distributions is close to uniform (in a sense which we will make precise later). This enables us to apply Lemma~\ref{lem:seqhoeff}, which gives us concentration results that imply our quasirandomness properties are maintained. A similar statement applies to the edge labels.

\subsection{Preparation}
\label{ssec:setup}

We now describe the general setup which we use for proving Theorem~\ref{thm:main}. We first reduce to the following special case.

\begin{theorem}\label{thm:mainred}
For every $\gamma>0$ such that $\gamma^{-1}$ is an integer, there exist $\Pc>0$ and $n_0\in\mathbb{N}$ such that the
following holds. Let $\mu=\big\lceil\exp\big(10^{8}\gamma^{-4}\big)\big\rceil^{-1}$, let $\delta_0=\mu^{2/\mu}$, and let $\delta=\delta_0^{10}$. For every $n>n_0$ divisible by $2\delta^{-1}\gamma^{-1}$, if $T$ is an $n$-vertex tree with
$\maxdeg(T)\le \tfrac{\Pc n}{\log n}$, there exists a
graceful labelling $\psi:V(T)\rightarrow [(1+\gamma)n]$.
\end{theorem}
(Note that the hierarchy of constants in Theorem~\ref{thm:mainred} is $\gamma>\mu> \delta_0>\delta>\Pc>0$.)

Before proving this, we show why it implies Theorem~\ref{thm:main}.

\begin{proof}[Proof of Theorem~\ref{thm:main}]
 Given $\gamma>0$, let $\gamma'=1/\lceil2\gamma^{-1}\rceil$. Let $\Pc>0$ and $n'_0\in\mathbb{N}$ be returned by Theorem~\ref{thm:mainred} for input $\gamma'$. Let $\mu=\big\lceil\exp\big(10^{8}\gamma'^{-4}\big)\big\rceil^{-1}$, let $\delta_0=\mu^{2/\mu}$, and let $\delta=\delta_0^{10}$. Let $n_0\ge n'_0$ be the smallest integer such that
 \[(1+\gamma')\big(n_0+2\delta^{-1}\gamma'^{-1}\big)\le (1+\gamma)n_0\]
 and such that $\tfrac{\Pc n}{\log n}\ge 2$. Given $n\ge n_0$, and an $n$-vertex tree $T$ with $\maxdeg(T) \le \tfrac{\Pc n}{\log n}$, let $n'$ be the smallest integer which is at least $n$ and  divisible by $2\delta^{-1}\gamma'^{-1}$. Let $T'$ be an $n'$-vertex tree obtained by attaching a path with $n'-n$ vertices to a leaf of $T$.  Theorem~\ref{thm:mainred} applies, so there is a graceful labelling of $T'$ with $(1+\gamma')n'\le(1+\gamma)n$ labels. The induced labelling of $T$ is also graceful, as desired.
\end{proof}

We now give the setup we will use to prove Theorem~\ref{thm:mainred}.

\begin{setup}\label{setup}
Given $\gamma>0$ such that $\gamma^{-1}$ is an integer, we choose\reminder{$\mu$, $\delta_0$, $\delta$, $\epsilon$, $\eta$} $\mu=\big\lceil\exp\big(10^{8}\gamma^{-4}\big)\big\rceil^{-1}$.
 We set $\delta_0=\mu^{2/\mu}$, $\delta=\delta_0^{10}$, $\eps=\delta^{10}$, and $\Pc=\eps^{10}$. Set
 \begin{equation}\label{eq:n0}
  n_0= 2^{\nicefrac{10}{\eta^2}}\,.
 \end{equation}
 Suppose now that $n>n_0$ divisible by $2\delta_0^{-10}\gamma^{-1}$ is given. Note that since $\mu^{-1}$ is an integer, $\delta_0^{-1}$ is also an integer, so this is possible.
 
 For each $0\le i\le n$ set
 \reminder{$\delta_i$} $\delta_i=\mu^{(2n-i)/(\mu n)}$. Observe that for $i=0$, this definition is consistent with the previous definition of $\delta_0$.
 
 Let \reminder{$\ell, m$}$\ell= \frac12 \delta_0^2 \gamma n$, and $m=\delta^2_0\ell=\frac12 \delta_0^4\gamma n$. Because $n$ is divisible by $2\delta_0^{-4}\gamma^{-1}$, and $\delta_0^{-1}$ is an integer, these two quantities are integers and $m$ divides $\ell$. 
 
 Let \reminder{$\tilde n$}$\tilde{n}=(1+\gamma)n$. Because $2\delta_0^{-4}\gamma^{-1}$ divides $n$, in particular $\tilde{n}$ is an integer multiple of $2m$.

Let \reminder{$\bbA, \bbC$} $\bbA:=\big[\tilde{n}\big]$, and $\bbC:=\big[\tilde{n}-1\big]$. We will choose vertex labels from $\bbA$, and edge labels from~$\bbC$.

Let \reminder{$\cI$} $\cI$ be the collection of intervals of length $m-1$ (i.e., size $m$) in $\bbA$ whose first elements are in the set
\[\{1,m+1,2m+1,\dots,\tilde{n}-m+1\}\,,\]
and let \reminder{$\cIe$} $\cIe$ be the collection of intervals of length $m-1$ in $\bbC\cup\{0\}$ whose first elements are in the set
\[\{0,m,\dots,|\bbC|-m+1\}\,.\]

Finally, let \reminder{$\cJ$}$\cJ$ be the set of intervals of length $\ell-1$ in $\bbA$ whose first elements are in the set
\[\big\{1,m+1,2m+1,\dots,\tfrac{\tilde{n}}{2}-m-\ell+1,\tfrac{\tilde{n}}{2}-\ell+1,\tfrac{\tilde{n}}{2}+1,\tfrac{\tilde{n}}{2}+m+1,\tfrac{\tilde{n}}{2}+2m+1,\dots,\tilde{n}-m-\ell+1,\tilde{n}-\ell+1\big\}\,.\]
(See Figure~\ref{fig:intervals}.)

For each $\iJ\in\cJ$, we define the \emph{complementary interval} \reminder{$\overline{\iJ}$}$\overline{\iJ}\in\cJ$ to be the (unique) interval such that the sum of the elements of $\iJ\cup\overline{\iJ}$ is $\ell(\tilde{n}+1)$. By definition, $\iJ$ and $\overline{\iJ}$ are disjoint, one having only elements less than or equal to $\tfrac12\tilde{n}$ and the other having only elements greater than $\tfrac12\tilde{n}$.
\end{setup}
\begin{figure}
\includegraphics[scale=0.8]{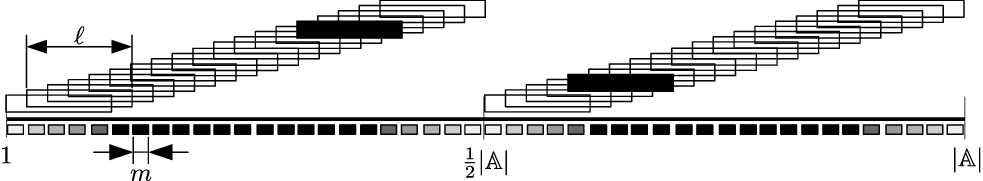}
\label{fig:intervals}
\caption{The intervals $\cI$ (below the line) and $\cJ$ (above the line). The shade of grey used for an interval $I\in\cI$ corresponds to the number of intervals of $\cJ$ containing $I$. A pair of complementary intervals of $\cJ$ is highlighted.}
\end{figure}
Note that
\begin{align}\label{eq:IJsize}
|\cI|=|\cIe|=\tfrac{\tilde{n}}{m} \quad\mbox{and}\quad |\cJ|=\tfrac{\tilde{n}}{m}-2\big(\tfrac{\ell}{m}-1\big)\;.
\end{align}

In analysing our labelling algorithm, the values $(\delta_i)_{i=1}^n$ will quantify errors in our quasirandomness properties in steps $i=1,\ldots,n$ (see Claim~\ref{cl:p}). Our choice of $\delta_i$ ensures that the following holds. 
\begin{fact}
For each $1\le t\le n$, we have
\begin{equation}\label{eq:deltas}
\sum_{i=1}^{\lceil\tfrac{t}{\delta n}\rceil}\delta \mu^{-1} \delta_{i\delta n}<\tfrac{1}{100}\delta_t\,.
\end{equation}	
\end{fact}
\begin{proof}
Let us write $a:=\mu^{-\delta/\mu}$. Later, we will later need that
\begin{equation}\label{eq:aaaautocz}
a\ge 2\;.
\end{equation}
We have
\begin{equation*}
a=\mu^{-\delta/\mu}=\mu^{-\mu^{\frac{20}{\mu}-1}}\;.
\end{equation*}
We shall just sketch~\eqref{eq:aaaautocz}, assuming that $\mu$ is `very small', i.e., we will look at the limit behaviour of $\mu$ around~0. It is easy (but tedious) to check that these calculations do go through for our $\mu$ (with a lot of room). Recall a basic fact from calculus that $\lim_{x\searrow 0}x^{\frac1x}=1$. Similar calculation yield $\lim_{x\searrow 0}x^{\frac{20}x-1}=1$. Hence, $a$ equals $\mu$ taken a power which is close to~$-1$, and thus $a$ is very big. In particular~\eqref{eq:aaaautocz} holds.

Next, let us bound $\sum_{i=1}^{\lceil\tfrac{t}{\delta n}\rceil}\mu^{-i\delta/\mu}$. Using the formula for geometric sequences, we obtain
\begin{align}\label{eq:bbbautocz}
\sum_{i=1}^{\lceil\tfrac{t}{\delta n}\rceil}\mu^{-i\delta/\mu}=\sum_{i=1}^{\lceil\tfrac{t}{\delta n}\rceil}a^i=\frac{a^{\lceil\tfrac{t}{\delta n}\rceil+1}-a}{a-1}\leByRef{eq:aaaautocz}\frac{a^{\lceil\tfrac{t}{\delta n}\rceil+1}}{\frac{a}{2}}=2a^{\lceil\tfrac{t}{\delta n}\rceil}\;.
\end{align}

Using the definition of $\delta_q$, we can expand the left-hand side as follows,
\begin{align*}
\sum_{i=1}^{\lceil\tfrac{t}{\delta n}\rceil}\delta \mu^{-1} \delta_{i\delta n}&=
\delta\mu^{-1}\cdot\mu^{2/\mu}\cdot\sum_{i=1}^{\lceil\tfrac{t}{\delta n}\rceil}\mu^{-i\delta/\mu}\\
&
\leByRef{eq:bbbautocz} \delta\mu^{-1}\cdot\mu^{2/\mu}\cdot 2\mu^{-\lceil t/(\delta n)\rceil\delta/\mu}=
2\delta\mu^{-1}\cdot\delta_{\lceil t/(\delta n)\rceil\delta n}\\
&\le 2\delta\mu^{-1}\cdot\mu^{-\delta/\mu}\cdot \delta_{t}
=2\delta \cdot \mu^{-(\frac{\delta}{\mu}+1)}\cdot \delta_t\;. 
\end{align*}
Recall that $\delta$ is much smaller than $\mu$. In particular, $\mu^{-(\frac{\delta}{\mu}+1)}<\mu^{-1.1}$. Using the relation between $\delta$ and $\mu$ once again, we get that $2\delta\cdot\mu^{-1.1}<\frac{1}{100}$. Hence, \eqref{eq:deltas} follows.
\end{proof}

The inequality~\eqref{eq:deltas} will be important in showing that our error terms do not grow too fast.

\subsection{Assigning intervals of labels to vertices}
Our next step is to give an order on $V(T)$ and for each vertex the promised `small interval' in $\cJ$ in which we will eventually choose its label. As mentioned in the outline of the proof, for most edges we will label the two endpoints from complementary small intervals. Specifically, we will do this for all edges not in the set $\mathcal{R}$ given by the following lemma.

\begin{lemma}\label{lem:setup} Assume Setup~\ref{setup}.
 Given an $n$-vertex tree $T$ with $\Delta(T)\le\frac{\Pc n}{\log n}$, there exists a set \reminder{$\mathcal{R}$}$\mathcal{R}\subset E(T)$, an ordering \reminder{$v_1,\ldots,v_n$}$V(T)=\{v_1,\ldots,v_n\}$ of the vertices of $T$, and a collection of intervals $\iJ(v_i)\in\cJ$ for each $1\le i\le n$ with the following properties.
 \begin{enumerate}[label=(PRE\arabic*)]
  \item\label{pre:sizeR} $|\mathcal{R}|\le\eps n$.
  \item\label{pre:degen} For each $i\ge 2$, the vertex $v_i$ has exactly one neighbour in the set $\{v_1,\ldots,v_{i-1}\}$. We call this vertex the \emph{parent of $v_i$}, and denote it by $\prt{v_i}$.\reminder{$\prt{v}$}
  \item\label{pre:int} If $v_iv_j\in E(T)\setminus\mathcal{R}$ then $|i-j|\le\frac{\eps n}{\log n}$.
  \item\label{pre:chJ} For each interval $S\subset[n]$ and each $\iJ\in\cJ$, we have $\sum_{i\in S}\ind_{\iJ(v_i)=\iJ}=\frac{|S|}{|\cJ|}\pm\delta^2 n$.
  \item\label{pre:comp} For each $v_iv_j\in E(T)\setminus\mathcal{R}$ we have $\iJ(v_i)=\overline{\iJ(v_j)}$.
 \end{enumerate}
\end{lemma}
\begin{proof}
 Given a tree $T$, let $\mathcal{R}$ be the set of edges of $T$ returned by Lemma~\ref{lem:cuttree} with input $\eps$. Then we have $|\mathcal{R}|\le\eps n$, giving~\ref{pre:sizeR}. Let $T_1,\ldots,T_s$ be the components of $T-\mathcal{R}$.

Let $v_1$ be an arbitrary vertex of $V(T)$. Now for each $i\ge 2$ in turn, we choose $v_i$ to be a vertex of $V(T)\setminus\{v_1,\ldots,v_{i-1}\}$ which has a neighbour in $\{v_1,\ldots,v_{i-1}\}$, if possible choosing a vertex in the same component of $T-\mathcal{R}$ as $v_{i-1}$.
\ref{pre:degen} follows from this construction.

Suppose now that in the above construction we cannot choose $v_i$ in the same component of $T-\mathcal{R}$ as $v_{i-1}$. Then this is because we already chose all vertices of that component. It follows that each component forms an interval in our ordering on the vertices. In particular, if $v_iv_j\in E(T)\setminus\mathcal{R}$ then $v_iv_j$ is in one component $T_k$ of $T-\mathcal{R}$. Then Lemma~\ref{lem:cuttree} tells us that $|i-j|\le v(T_k)\le\frac{\epsilon n}{\log n}$, giving~\ref{pre:int}.

We now take an arbitrary proper colouring of $T$ with two colours, red and blue. For each component $T_k$ of $T-\mathcal{R}$ we choose independently an interval $\iJ(T_k)$ uniformly at random from $\cJ$. For each red $v_j\in T_k$ we set $\iJ(v_j)=\iJ(T_k)$, and for each blue $v_j\in T_k$ we set $\iJ(v_j)=\overline{\iJ(T_k)}$. This gives~\ref{pre:comp} deterministically.

It remains to show that with positive probability we obtain~\ref{pre:chJ}. To that end, let $S\subset[n]$ be an interval, and $\iJ$ be an element of $\cJ$. We view $S$ as an interval in $(v_1,\dots,v_n)$. Consider the intersections of various components $T_k$ with $S$. Since these components form intervals in $(v_1,\dots,v_n)$, at most two components have non-empty intersection with $S$ and are not contained in $S$. These at most two components contain at most $2\tfrac{\eps n}{\log n}<\delta^5 n$ vertices. Let $\mathcal K$ index the components which are contained in $S$, i.e., $T_k\subset S$ for each $k\in\mathcal K$. For each $k\in\mathcal K$ define $Y_k$ as follows. If $\iJ(T_k)=\iJ$, let $Y_k$ be equal to the number of red vertices in $T_k$; if $\overline{\iJ(T_k)}=\iJ$ let $Y_k$ be equal to the number of blue vertices in $T_k$, and otherwise let $Y_k=0$.
  
  Because the sets $\iJ(T_k)$ are chosen independently, we get that $Y_k$ are independent random variables, with $0\le Y_k\le v(T_k)$ for each $k\in\mathcal K$. We have $\Exp[Y_k]=\frac{v(T_k)}{|\cJ|}$ for each $k\in\mathcal K$. Because $v(T_k)\le\tfrac{\eps n}{\log n}$ for each $k\in\mathcal K$, and $\sum_{k\in\mathcal K}v(T_k)\le n$, we have
  \[
\sum_{k\in\mathcal K}v(T_k)^2\le\frac{n}{\;\tfrac{\eps n}{\log n}\;}\cdot\left(\frac{\eps n}{\log n}\right)^2=\frac{\eps n^2}{\log n}\,.\]
  Let $X=\sum_{k\in\mathcal K}Y_k$. By Theorem~\ref{thm:hoeff} we have
  \[\Prob\big[|X-\Exp[X]|>\tfrac{1}{4}\delta^4 n\big]\le 2\exp\left(-\tfrac18\cdot\delta^{8}n^2\cdot\tfrac{\log n}{\eps n^2}\right)=2\exp\left(-\tfrac{1}{8\delta^2}\log n\right)\le n^{-10}\,, \]
  where we used the relation between $\epsilon$ and $\delta$ from Setup~\ref{setup}.
  
  We see that with probability at least $1-n^{-10}$ we have
  \[X=\frac{|S|\pm\delta^5 n}{|\cJ|}\pm\tfrac14\delta^4 n=\frac{|S|}{|\cJ|}\pm\tfrac12\delta^4 n\,.\]
  Since the number of vertices $v_i\in S$ with $\iJ(v_i)=\iJ$ is $X$, we conclude that for the chosen $S$ and $\iJ$ the probability that~\ref{pre:chJ} fails is at most $n^{-10}$. Taking the union bound over $|\cJ|\le n$ choices of $\iJ$ and $\binom{n}{2}$ choices of $S$ we get that with probability at least $1-n^{-7}$ we have~\ref{pre:chJ}, as desired.
\end{proof}

\subsection{Counting structures, and quasirandom properties}
\label{ssec:quasirandomproperties}
The phenomenon of quasirandomness is central in discrete mathematics. For example, the celebrated Chung--Graham--Wilson Theorem asserts that if the edge density and the four-cycle density of an $n$-vertex graph $G$ are close to those of an Erd\H{o}s--R\'enyi random graph $\mathbb{G}(n,p)$, then $G$ has many other properties as a typical $\mathbb{G}(n,p)$.
In the proof of Theorem~\ref{thm:mainred}, we need to control the evolution of the sets $\setVert\subset \bbA$ and $\setEdge\subset \bbC$ of vertex labels and edge labels not used so far during the sequential labeling of $T$. We want to prove that the pair $(\setVert,\setEdge)$ behaves quasirandomly. Thus, in some analogy to the Chung--Graham--Wilson Theorem, we want to come up with quantities control over of which will imply further quasirandomness properties of $(\setVert,\setEdge)$. To this end, we introduce a class $\cX$ of `structures' in Section~\ref{sssec:structures}. We will write $|X(\setVert,\setEdge)|$ for the number of structures of a given type which appear in $(\setVert,\setEdge)$, and the main technical work of this paper will be to show that this number remains close to what one would expect if the two sets were chosen independently at random. In Section~\ref{sssec:roleofX} we then explain that these parameters are indeed useful for our graceful labelling. In Section~\ref{sssec:quasirandomness} we then state our main quasirandomness condition and state a useful lemma connected to it.

\subsubsection{Structures}\label{sssec:structures}
The key objects that allow us to control quasirandomness are `structures' defined below.
\begin{definition}[structure]\label{def:struct}
 A \emph{structure} $X$ is a graph such that
 \begin{itemize}
 	\item its vertices are labelled with either pairwise distinct elements of $\bbA$ (we call such vertices \emph{fixed}) or pairwise distinct intervals in $\bbA$ (we call such vertices \emph{free}), and
 	\item its edges are labelled with either pairwise distinct elements of $\bbC$ (\emph{fixed}) or with distinct choices of special symbols $\ed_1$ or $\ed_2$ (\emph{free}). 
\end{itemize} 	
 	When dealing with structures, we identify vertices and edges with their labels. So, vertices in $X$ are numbers (if they are fixed) or intervals (if they are free). Likewise, edges in $X$ are numbers or special free symbols $\ed_1$ and $\ed_2$.
 
In any structure $X$, we require that if $u$ and $v$ are fixed vertices, then $uv$ is a fixed edge and we have $uv=|u-v|$, and we require that each free edge has one endpoint fixed and the other free.
\end{definition}
We shall be interested in four groups of structures, denoted by $X_1,\ldots,X_4$. Actual structures in each individual group, say in $X_i$, have the same underlying graph, but differ by labels on vertices and edges. So, let us first describe the graphs underlying these four groups of structures. The graph underlying $X_1$ is a single vertex. The graph underlying $X_2$ and $X_4$ is a path on three vertices. The graph underlying $X_3$ is an edge. Let us now describe the free and fixed vertices and edges of these four groups of structures. The single vertex of $X_1$ is free. Members of $X_2$ have one end-vertex fixed and the two remaining vertices free. The edge connecting the two free vertices is fixed and the other one is free. In $X_3$, one vertex is free and the other is fixed; the edge connecting these two vertices is free. Last, the center of $X_4$ is free, the end-vertices are fixed, and the two edges are free. We then refer to the actual structures in group $X_i$ by writing $X_i\llbracket\cdot\rrbracket$ where the double brackets contains specification of labels of fixed vertices (these are specified by an element of $\bbA$), free vertices (these are specified by an interval in $\bbA$), and fixed edges (these are specified by an element of $\bbC$). Note that free edges are not parametrized in Definition~\ref{def:struct}, and hence no information regarding them is included the double brackets.

For example, individual structures within the group $X_2$ differ by the actual label on the fixed vertex, the two labels on the free vertices, and the label on the single fixed edge. 
That is, given $a,a'\in\bbA$, $c\in\bbC$ and distinct $I,I'\in\cI\cup\cJ$ \reminder{$\Xone{I}$, $\Xtwo{a,I,c,I'}$, $\Xthree{a,I}$, $\Xfour{a,a',I}$,}
we write
$\Xone{I}$, $\Xtwo{a,I,c,I'}$, $\Xthree{a,I}$ and $\Xfour{a,a',I}$ for structures as shown on
Figure~\ref{fig:structs}. 

\begin{definition}\label{def:cX}
Let $\cX$ \reminder{$\cX$} be the set of structures of the form $\Xone{I}$, $\Xtwo{a,I,c,I'}$, $\Xthree{a,I}$ and $\Xfour{a,a',I}$ where $a,a'$ are distinct elements of $\bbA$, where $c$ is an element of $\bbC$, and where $I,I'\in\cI$ are distinct intervals.
\end{definition}
\begin{figure}[t]
\caption{Structures for quasirandomness}
\label{fig:structs}
\begin{tikzpicture}[thick,main node/.style={circle,minimum size=9mm,draw}]
  \node[main node] at (0,0) (1) {$I$};
  \node[align=center] at (0,-1) {$\Xone{I}$};
  \node[main node] at (3,0) (2) {$a$};
  \node[main node] at (4,2) (3) {$I$};
  \node[main node] at (5,0) (4) {$I'$};
  \path
   (2) edge node[left] {$\ed_1$} (3)
   (3) edge node[left] {$c$} (4);
  \node[align=center] at (4,-1) {$\Xtwo{a,I,c,I'}$};
  
  \node[main node] at (8,0) (5) {$a$};
  \node[main node] at (8,2) (6) {$I$};
  \path
   (5) edge node[left] {$\ed_1$} (6);
  \node[align=center] at (8,-1) {$\Xthree{a,I}$}; 
  
  \node[main node] at (11,0) (7) {$a$};
  \node[main node] at (12,2) (8) {$I$};
  \node[main node] at (13,0) (9) {$a'$};
  \path
   (7) edge node[left] {$\ed_1$} (8)
   (8) edge node[left] {$\ed_2$} (9);
  \node[align=center] at (12,-1) {$\Xfour{a,a',I}$}; 
  
\end{tikzpicture}
\end{figure}

\begin{remark}\label{rem:polibmi}
	We shall explain in Section~\ref{sssec:roleofX} the essential role structures play in our proof. But at this moment, let us give at least a brief hint behind their definition. Say, we are dealing with labeling of particular three vertices $v_1$, $v_2$, and $v_3$ of our tree, which are in a mutual position as in group $X_2$. Say, for some reason (which we explain in Section~\ref{sssec:roleofX}), we want to label $v_1$ with $a$, $v_2$ with a label from $I$ and $v_3$ with a label from $I'$. Further, suppose that we want the induced edge label on $v_2v_3$ to be $c$. Then these requirements are clearly reflected by $\Xtwo{a,I,c,I'}$. You can also observe that such a situation has exactly `one degree of freedom', where each degree of freedom represents the possibility of many choices (more than 2) of one label. Indeed, selecting arbitrarily a label $a_2$ for $v_2$ in $I$, we have at most two options, namely $a_2+c$ and $a_2-c$ (these numbers need not be in $I'$). In this sense, all structures $\Xone{I}$, $\Xtwo{a,I,c,I'}$, $\Xthree{a,I}$ and $\Xfour{a,a',I}$ have exactly one degree of freedom.
	
	This perspective also explains why a free edge is not specified by any interval, in contrast to free vertices. Indeed, each free edge in the above structures is incident to a fixed vertex from one side and to a free vertex from another side, and hence specifying an interval for the free vertex already tells us the interval for the induced edge label. The only reason why we introduce the `special symbols' $\ed_1$ and $\ed_2$ is to distinguish the two edges graph-theoretically (that is, without taking into account different properties their induced edge labels may enjoy).
	
	Last, let us explain why Definition~\ref{def:cX} we require $I$ and $I'$ to be distinct. Recall that in Lemma~\ref{lem:setup}\ref{pre:comp} the intervals assigned to the two endvertices of each edge $T$ (except the special ones in $\mathcal{R}$) are complementary, and hence distinct. $I$ and $I'$ will represent these assigned intervals.
\end{remark}

We can now define define a concept of ``one structure following the pattern of another structure'' (so, we do not define any notion of ``pattern'' per se), and of chosen labels.

\begin{definition}[to follow pattern, chosen label]\label{def:chosenEtAl}
 Suppose that we have a structure $X$. Suppose that $X$ has a free edge $\ed_i$ with fixed endpoint labelled $a$ and free endpoint labelled $I$. We then write
\reminder{$\diff(\ed_i;X)$}
$\diff(\ed_i;X):=|a-\min(I)|$. We write $\freev(X)$\reminder{$\freev(X)$} for the set of free vertex labels in $X$, and $\freee(X)$ \reminder{$\freee(X)$}for the set of free edge labels in $X$.  We write $\free(X)$\reminder{$\free(X)$} for the total number of free (vertex- or edge-) labels in $X$. 

Suppose that we are given structures $X$ and $X'$. We say that $X'$ \emph{follows the pattern} $X$ if $\free(X')=0$ and there is a graph isomorphism $\rho$ from $X'$ to $X$ such that 
\begin{itemize}
	\item for each vertex $u\in X'$ for which $\rho(u)$ is fixed in $X$, the labels of $u$ and of $\rho(u)$ are the same,
	\item for each vertex $u\in X'$ for which $\rho(u)$ is free in $X$, the label of $u$ is contained in the label of $\rho(u)$,
	\item for each edge $uv\in X'$ for which $\rho(uv)$ is a fixed edge in $X$, the labels of $uv$ and of $\rho(uv)$ are the same.
\end{itemize}
 We call the labels of vertices $u\in X'$ for which $\rho(u)$ is free in $X$, \emph{chosen (vertex) labels}. Likewise, we call the labels of edges $uv\in X'$ for which $\rho(uv)$ is free in $X$, \emph{chosen (edge) labels}. We write $\chosenv{X'}{X}$\reminder{$\chosenv{X'}{X}$} for the set of chosen vertex labels in $X'$ and $\chosene{X'}{X}$\reminder{$\chosene{X'}{X}$} for the set of chosen edge labels in~$X'$.

Given $\setVert\subset\bbA$ and $\setEdge\subset\bbC$, and a structure $X$, we write
\reminder{$X(\setVert,\setEdge)$}
\begin{equation}\label{eq:substitutionbrackets}
X(\setVert,\setEdge)
\end{equation} for the set of all structures following the pattern $X$ whose chosen vertex labels are in $\setVert$ and whose chosen edge labels are in $\setEdge$.
\end{definition}
Note that the double square brackets use to parametrise the families $X_1,\ldots,X_4$ have a different meaning than the parentheses in~\eqref{eq:substitutionbrackets}; we can for example write $\Xtwo{a,I,c,I'}(\setVert,\setEdge)$ for the set of all structures following the pattern $\Xtwo{a,I,c,I'}$ whose chosen vertex labels are in $\setVert$ and whose chosen edge label is in~$\setEdge$.

\subsubsection{The role of structures $\cX$}\label{sssec:roleofX}
At the beginning of Section~\ref{sec:setup} we outlined the main idea of the proof of Theorem~\ref{thm:main} (which is also the main idea of  Theorem~\ref{thm:mainred}): We proceed labeling the vertices of $T$. During this process, the set $\setVert\subset \bbA$ of available vertex labels and the set $\setEdge\subset \bbC$ of available edge labels get sparser. We need to control that the sets $\setVert$ are $\setEdge$ are spread over the intervals $\bbA$ and $\bbC$  in a quasirandom way. Actually, we need to control even the interactions of $\setVert$ and $\setEdge$; for example if $\setVert$ consisted of all even numbers and $\setEdge$ of all odd numbers then these sets are very uniformly spread, but it is not possible to label a single new pair of vertices that form an edge of $T$. It turns out that the quantities we need to control for our proof of Theorem~\ref{thm:mainred} are exactly the quantities $|X(\setVert,\setEdge)|$, for each $X\in \cX$. In other words, we control the number of structures that follow the pattern $X$ and use elements from $\setVert$ and $\setEdge$ as the chosen vertex labels and edge labels, respectively.
For example, observe that the density $\frac{|\setVert\cap I|}{|I|}$ of $\setVert$ on $I\in\cI$ is equal to $\tfrac{|\Xone{I}(\setVert,\setEdge)|}{|\Xone{I}(\bbA,\bbC)|}$. For the other three structures, one should think of $a\in\bbA$ as the chosen label for some vertex $v_i$ of $T$. Then $|\Xthree{a,I}(\setVert,\setEdge)|$ is the number of ways to give a neighbour $v_j$ of $v_i$ a label in $I$ which has not yet been used and which induces an edge label that has not yet been used (which is obviously a useful quantity to control). If one thinks of $a$ as being the chosen label of a vertex $v_i$, then $|\Xtwo{a,\iJ(v_j),c,\iJ(v_k)}(\setVert,\setEdge)|$ is the number of ways to label a child $v_j$ of $v_i$ and grandchild $v_k$ of $v_i$ within their chosen intervals, not re-using vertex or edge labels used previously, and using the edge label $c$ on $v_jv_k$. The quantity $|\Xfour{a,a',\iJ(v_j)}(\setVert,\setEdge)|$ plays a similar r\^ole, except that we fix the vertex label used for $v_k$ to be $a'$ rather than the edge label for $v_jv_k$.

One of our quasirandomness conditions,~\ref{quasi:X} below, states that these quantities are likely to stay close to what one would expect if the sets $\setVert$ and $\setEdge$ were chosen independently at random; thus $|\Xone{I}(\setVert,\setEdge)|$ will be close to the overall density of $\setVert$ in $\bbA$ when $\setVert$ is generated by running our labelling algorithm for a given time. In addition, we introduce property~\ref{quasi:edged}, which states that the density of $\setEdge$ on each interval $\iIe\in\cIe$ is approximately $\tfrac{|\setEdge|}{|\bbC|}$.

\medskip

Since this idea of a structure following the pattern of another is a little complicated, before stating Fact~\ref{fact:zada} which elaborates on our remark about one degree of freedom from Remark~\ref{rem:polibmi}, and puts structural and quantitative restrictions on $X(\setVert,\setEdge)$ for any structure $X\in\cX$ and sets $\setVert$ and $\setEdge$, we give a little example.
\begin{example}\label{ex:zada}
In Figure~\ref{fig:X2specificvertex} we give an example of structures with a specific label on a specific free vertex or edge. More precisely, we show the only four structures following $\Xtwo{10,I,3,I'}$ (for say $I=I'=\{5,6,\dots,15\}$) which induce label $2$ on $\ed_1$. The two leftmost structures are also the only structures following $\protect\Xtwo{10,I,3,I'}$ that have label 12 on $I$. Since $I$ and $I'$ overlap, they cannot be two distinct sets of $\cI$, and thus $\Xtwo{10,I,3,I'}\not\in\cX$, i.e.\ this particular structure is not one we will be interested in. We are only interested in $\Xtwo{10,I,3,I'}\in \cX$, which means $I$ and $I'$ are disjoint. If $10$ is in $I$, we might still have two possible vertex labels in $I$ which induce label $2$ on $\ed_1$ (if $10$ is not in $I$, there will be at most one such choice), but given such a vertex label there is at most one valid choice of vertex label in $I'$.
\end{example}

\begin{figure}[t]
	\caption{The four structures following $\protect\Xtwo{10,I,3,I'}$ (for say $I=I'=\{5,6,\dots,15\}$) which induce label 2 on $\ed_1$. The two leftmost structures are also the only structures following $\protect\Xtwo{10,I,3,I'}$ that have label 12 on $I$.}
	\label{fig:X2specificvertex}
	\begin{tikzpicture}[thick,main node/.style={circle,minimum size=9mm,draw}]
	\node[main node] at (3,0) (2) {$10$};
	\node[main node] at (4,2) (3) {$12$};
	\node[main node] at (5,0) (4) {$9$};
	\path
	(2) edge node[left] {$2$} (3)
	(3) edge node[left] {$3$} (4);
	
	\node[main node] at (11,0) (2) {$10$};
	\node[main node] at (12,2) (3) {$8$};
	\node[main node] at (13,0) (4) {$5$};
	\path
	(2) edge node[left] {$2$} (3)
	(3) edge node[left] {$3$} (4);
	
	\node[main node] at (7,0) (2) {$10$};
	\node[main node] at (8,2) (3) {$12$};
	\node[main node] at (9,0) (4) {$15$};
	\path
	(2) edge node[left] {$2$} (3)
	(3) edge node[left] {$3$} (4);
	
	\node[main node] at (15,0) (2) {$10$};
	\node[main node] at (16,2) (3) {$8$};
	\node[main node] at (17,0) (4) {$11$};
	\path
	(2) edge node[left] {$2$} (3)
	(3) edge node[left] {$3$} (4);
	\end{tikzpicture}
\end{figure}

We now state the following fact, which generalises the above example.

\begin{fact}\label{fact:zada}
Suppose that $X\in\cX$ and sets $\setVert$ and $\setEdge$ are given. 
\begin{enumerate}[label=(\alph*)]
\item\label{en:ZadaA}
Given any free vertex and a label in $\bbA$ there is at most one structure in $X(\setVert,\setEdge)$ in which the corresponding vertex has that label. Similarly for any free edge and label in $\bbC$ there are at most two structures in $X(\setVert,\setEdge)$ in which the corresponding edge has that label.
\item\label{en:ZadaB}
We have $|X(\setVert,\setEdge)|\le \min_{I\in\freev(X)}|I|=m$.
\end{enumerate}
\end{fact}
\begin{proof}
We first deal with Part~\ref{en:ZadaA}, the case of substituting a chosen label, say $a^*$, on a free vertex. The statement is trivial for structures of the groups $X_1$, $X_3$ and $X_4$, because for these, substituting $a^*$ on the originally free vertex fully specifies that structure. For structures $\Xtwo{a,I,c,I'}$, we recall that Definition~\ref{def:cX} requires all elements of $I$ to be smaller than all elements of $I'$ or vice versa. So, if $a^*$ is the chosen vertex label for the smaller of the two intervals, then we know that the bigger of the two must get the label $a^*+c$. If $a^*$ is the chosen vertex label on the bigger of the two intervals, then we know that the smaller of the two must get the label $a^*-c$.

When a chosen edge label, say $c^*$, is substituted in structures $\Xtwo{a,I,c,I'},\Xthree{a,I},\Xfour{a,a',I}$ on the free edge $\ed_1$, then in all these cases the free vertex labeled with $I$ must be given a chosen vertex label $a+c^*$ or $a-c^*$. This completely specifies the chosen labels in case of $\Xthree{a,I},\Xfour{a,a',I}$. In case of $\Xtwo{a,I,c,I'}$, there is at most one way of choosing a vertex label for the free vertex labeled $I'$ by the argument from the previous paragraph. The free edge $\ed_2$ appears only in structure $\Xfour{a,a',I}$, where the argument is completely symmetric.

Part~\ref{en:ZadaB} follows from Part~\ref{en:ZadaA}.
\end{proof}

\subsubsection{The quasirandomness conditions}\label{sssec:quasirandomness}
As explained in Section~\ref{sssec:roleofX}, our quasirandomness conditions, given in Definition~\ref{def:quasirandom} below, assert that the pair $(\setVert,\setEdge)$ is a quasirandom subset of the pair $(\bbA,\bbC)$ of density roughly $\frac{|\setVert|^2}{|\bbA|^2}$. Here, the quasirandomness condition is expressed via structures from $\cX$, that is, by comparing $|X(\setVert,\setEdge)|$ and $|X(\bbA,\bbC)|$ for $X\in \cX$. For this reason, it is convenient to express the quantities $|X(\bbA,\bbC)|$ explicitly. This is done in the easy lemma below.
\begin{lemma}\label{lem:expressXcount}
	Suppose that $a$ and $a'$ are distinct elements of $\bbA$, that $c\in\bbC$, that $I\in\cI$, and that $\iJ\in\cJ$. Then we have the following.
	\begin{enumerate}[label=(\alph*)]
		\item\label{en:fullX3} $\big|\Xthree{a,\iJ}(\bbA,\bbC)\big|=\ell\pm 1$,
		\item\label{en:fullX3cI} $\big|\Xthree{a,I}(\bbA,\bbC)\big|=m\pm 1$,
		\item\label{en:fullX4} $\big|\Xfour{a,a',\iJ}(\bbA,\bbC)\big|=\ell\pm 3$,
		\item\label{en:fullX2}
		$\big|\Xtwo{a,\iJ,c,\overline{\iJ}}(\bbA,\bbC)\big|=\big|\{(b,b')\in\iJ\times\overline{\iJ}:|b-b'|=c\}\big|\pm 3$.
	\end{enumerate}
\end{lemma}
\begin{proof}
For~\ref{en:fullX3}, observe that $\Xthree{a,\iJ}(\bbA,\bbC)$ contains all structures that follow the pattern $\Xthree{a,\iJ}$ and whose single chosen vertex is labelled with an element of $\iJ$. The only exception is when $a\in\iJ$; in that case the structure in which the chosen vertex is labelled with $a$ is not counted (recall that in Definition~\ref{def:struct}, vertices are labelled with distinct labels). Since $|\iJ|=\ell$, we have that $\big|\Xthree{a,\iJ}(\bbA,\bbC)\big|=\ell$ or $\big|\Xthree{a,\iJ}(\bbA,\bbC)\big|=\ell-1$. The proof of~\ref{en:fullX3cI} is exactly the same.

Similarly, for~\ref{en:fullX4}, observe that $\Xfour{a,a',\iJ}(\bbA,\bbC)$ contains all structures that follow the pattern $\Xfour{a,a',\iJ}$ and whose chosen vertex is labelled by any element from $\iJ\setminus \{a, a', \tfrac{a+a'}{2}\}$ (the reason for excluding $\tfrac{a+a'}{2}$ is that such a choice would give us the same label on the edges $\ed_1$ and $\ed_2$).

  For~\ref{en:fullX2}, observe that $\Xtwo{a,\iJ,c,\overline{\iJ}}(\bbA,\bbC)$ contains all structures that follow the pattern $\Xtwo{a,\iJ,c,\overline{\iJ}}$ and whose label $b$ chosen on the free vertex $\iJ$ and label $b'$ chosen on the free vertex $\overline{\iJ}$ satisfy 
  \begin{enumerate}[label=(\roman*)]
  	\item\label{en:lc1}$b\in \iJ$, $b'\in\overline{\iJ}$, $|b-b'|=c$, and
  	\item\label{en:lc2}$b\neq a$, $b'\neq a$, $|b-a|\neq|b-b'|$.
  \end{enumerate}
The number of pairs $(b,b')$ satisfying~\ref{en:lc1} is $\big|\{(b,b')\in\iJ\times\overline{\iJ}:|b-b'|=c\}\big|$, and each of the three restrictions in~\ref{en:lc2} decreases this number by at most~1.
\end{proof}

\begin{definition}[Quasirandomness conditions]\label{def:quasirandom}
\label{def:qr-conditions}
A pair of sets $\setVert\subset\bbA$ and $\setEdge\subset\bbC$ is \emph{$\alpha$-quasirandom} if we have the following.
\begin{enumerate}[label=(QUASI\arabic*)]
\item\label{quasi:edged} For all $\iIe\in\cIe$ we have $|\iIe\cap \setEdge|=m\cdot\frac{|\setVert|}{\tilde{n}}\pm\alpha m$.
\item\label{quasi:X} For all $X \in \cX$ we have $|X(\setVert,\setEdge)|=|X(\bbA,\bbC)|\big(|\setVert|/\tilde{n}\big)^{\free(X)}\pm\alpha m$.
\end{enumerate}
\end{definition}

 It may seem strange that we only have $|\setVert|$ in the estimates, and not $|\setEdge|$, but we will prove that throughout our labelling process, $|\setVert|$ and $|\setEdge|$ are close enough for the difference to be immaterial. Note that we only insist on counts of structures in $\cX$ being preserved, so that $I,I'\in\cI$, even though we defined our structures allowing $I,I'\in\cI\cup\cJ$. The following claim lets us deduce the latter (in the cases we need it) from the former.
 
 \begin{lemma}\label{lem:struct} Suppose that $(\setVert,\setEdge)$ is $\alpha$-quasirandom, that $a$ and $a'$ are distinct elements of $\bbA$, that $c\in\bbC$, and that $\iJ\in\cJ$. If $\ell\ge 3\alpha^{-1}$, then we have
\begin{align}
    \label{XtwoiJ}\big|\Xtwo{a,\iJ,c,\overline{\iJ}}(\setVert,\setEdge)\big|&=\big(|\setVert|/\tilde{n}\big)^3\big|\{(b,b')\in\iJ\times\overline{\iJ}:|b-b'|=c\}\big|\pm 3\alpha\ell\,, \\
    \label{XthreeiJ}\big|\Xthree{a,\iJ}(\setVert,\setEdge)\big|&=\big(|\setVert|/\tilde{n}\big)^2\ell\pm 2\alpha\ell\quad\text{and} \\
    \label{XfouriJ}\big|\Xfour{a,a',\iJ}(\setVert,\setEdge)\big|&=\big(|\setVert|/\tilde{n}\big)^3\ell\pm 2\alpha\ell\,.
  \end{align}
 \end{lemma}
 \begin{proof}
Let us first establish~\eqref{XthreeiJ}. Let us recall that the interval $\iJ$ is partitioned into $\ell/m$ many intervals $\{I\}_{I\in\cI,I\subset\iJ}$. Therefore, $\Xthree{a,\iJ}(\setVert,\setEdge)$ is partitioned into $\left\{\Xthree{a,I}(\setVert,\setEdge)\right\}_{I\in\cI,I\subset\iJ}$. Property~\ref{quasi:X} applies to each structure $\Xthree{a,I}$ with $I\in\cI$ and $I\subset\iJ$. Hence,
  \begin{align*}
    \big|\Xthree{a,\iJ}(\setVert,\setEdge)\big|&=\sum_{I\in\cI,I\subset\iJ}\big|\Xthree{a,I}(\setVert,\setEdge)\big|\\
  &=\big(|\setVert|/\tilde{n}\big)^2\sum_{I\in\cI,I\subset\iJ}\left(\big|\Xthree{a,I}(\bbA,\bbC)\big|\pm \alpha m\right) \\
  &=\big(|\setVert|/\tilde{n}\big)^2\big|\Xthree{a,\iJ}(\bbA,\bbC)\big|\pm \alpha\ell\\
  \JUSTIFY{by Lemma~\ref{lem:expressXcount}\ref{en:fullX3}}&=\big(|\setVert|/\tilde{n}\big)^2\ell\pm 2\alpha\ell\;.
  \end{align*}

  The proof of ~\eqref{XfouriJ} is very similar:
  \begin{align*}
\big|\Xfour{a,a',\iJ}(\setVert,\setEdge)\big|&=\sum_{I\in\cI,I\subset\iJ}\big|\Xfour{a,a',I}(\setVert,\setEdge)\big|\\
& =\big(|\setVert|/\tilde{n}\big)^3\sum_{I\in\cI,I\subset\iJ}\left(\big|\Xfour{a,a',I}(\bbA,\bbC)\big|\pm \alpha m\right)\\
 &=\big(|\setVert|/\tilde{n}\big)^3\big|\Xfour{a,a',J}(\bbA,\bbC)\big|\pm \alpha\ell\\
 \JUSTIFY{by Lemma~\ref{lem:expressXcount}\ref{en:fullX4}}&=\big(|\setVert|/\tilde{n}\big)^3\ell\pm 2\alpha\ell\;. 
  \end{align*}
   
   For~\eqref{XtwoiJ}, observe that
  \[\big|\Xtwo{a,\iJ,c,\overline{\iJ}}(\setVert,\setEdge)\big|=\sum_{\substack{I,I'\in\cI\\I\subset\iJ,I'\subset\overline{\iJ}}}\big|\Xtwo{a,I,c,I'}(\setVert,\setEdge)\big|\]
  and furthermore in the sum at most $2\ell/m$ terms are non-zero. Summing~\ref{quasi:X}, and observing that we only need to sum the $\alpha m$ error term $2\ell/m$ times, we obtain
\begin{align*}
\big|\Xtwo{a,\iJ,c,\overline{\iJ}}(\setVert,\setEdge)\big|&=\big(|\setVert|/\tilde{n}\big)^3\big|\Xtwo{a,\iJ,c,\overline{\iJ}}(\bbA,\bbC)\big|\pm 2\alpha\ell\\
\JUSTIFY{by Lemma~\ref{lem:expressXcount}\ref{en:fullX2}}&=\big|\{(b,b')\in\iJ\times\overline{\iJ}:|b-b'|=c\}\big|\pm 3\alpha\ell\,.
\end{align*}
 \end{proof}

\subsection{The algorithm}

The idea of the labelling algorithm is now straightforward. We will label the vertices in order, choosing at each time $t$ to give $v_t$ a vertex label in $\iJ(v_t)$ which has not previously been used, and which induces an edge label on $v_t\prt{v_t}$ which has not previously been used. Unfortunately, this simple version of the algorithm does not quite maintain the quasirandom properties mentioned above, because different intervals $I\in\cI$ are contained in different numbers of intervals $\iJ\in\cJ$; this is the `technical complication' mentioned in Footnote~\ref{foot:technicalcomplication}. Indeed, in Figure~\ref{fig:intervals} we saw that the labels around the extremes and the centre of $\bbA$ are used less frequently than those in the intermediate ranges. To correct this we introduce a distribution $\Corv$ on $\cI\cup\{\ast\}$. At each time we in addition randomly sample from $\Corv$, and either do nothing (if $\Corv$ returns $\ast$) or remove a randomly chosen so far unused vertex label from $I$ (if $\Corv$ returns $I$).

Most of the mass of $\Corv$ is on $\ast$. That means that the total number of labels removed during the run will be tiny compared to the `extra' $\gamma n$ labels we are given by Theorem~\ref{thm:mainred}. On the other hand, near the extremes and the centre of $\bbA$ a substantial proportion of the vertex labels will be removed by $\Corv$ without serving as vertex labels. These two contrasting properties are consistent since only a small number of vertex labels are near the extremes and the centre of $\bbA$.

An analogous complication arises when dealing with  edge labels $\bbC$. Indeed, we can see, for example, that for $\tilde{n}-1\in\bbC$ to appear as an edge label on $uv$, we must have chosen $\iJ(u)$ and $\iJ(v)$ to be the two extreme intervals of $\cJ$, while for $\tfrac{n}{2}\in\bbC$ (or any other edge label not close to $0$ or $\tilde{n}$) there are $\frac{2\ell^2}{m^2}$ possible choices of the pair $\iJ(u)$, $\iJ(v)$. To deal with these discrepancies, we introduce a suitable distribution $\Core$ on $\cIe\cup\{\ast\}$. In analogy with $\Corv$, we remove edge vertex labels from intervals $\iIe$ chosen according to $\Core$.

We will see that our labelling algorithm labels $v_i$ more or less uniformly in $\iJ=\iJ(v_i)$. If $v_i\prt{v_i}$ is not in $\mathcal{R}$, then the induced edge label on $v_i\prt{v_i}$ is chosen (approximately) from the distribution in which the probability of choosing $c$ is
\reminder{$\el{\iJ,c}$}
\begin{equation}\label{eq:AmiSlips}
\el{\iJ,c}:=\frac{\big|\big\{(a,a')\in\iJ\times\overline{\iJ}:|a-a'|=c\big\}\big|}{\ell^2}\,.
\end{equation}
For convenience, we define $\el{\iJ,0}$ according to the above formula (even though $0$ is not an edge label). 

We shall need the following simple properties of $\el{\iJ,c}$.
\begin{fact}\label{fact:alexandria}
\begin{enumerate}[label=(\alph*)]
	\item\label{en:Ocasio} 
For any edge label $c$, there are at most $\nicefrac{2}{\delta_0^2}$ sets $\iJ\in\cJ$ such that $\el{\iJ,c}>0$.
\item\label{en:Cortez}For any edge labels $c$ and $c'$ with $|c-c'|\le 2m$, and any $\iJ\in\cJ$ we have $\big|\el{\iJ,c}-\el{\iJ,c'}\big|\le\frac{2m}{\ell^2}$. 
\end{enumerate}	
\end{fact}
\begin{proof}
Part~\ref{en:Ocasio} is obvious. For part~\ref{en:Cortez}, we can assume that $\iJ$ precedes $\overline{\iJ}$, and that $c\le c'$. Lets expand the nominators in~\eqref{eq:AmiSlips} corresponding to $\el{\iJ,c}$ and $\el{\iJ,c'}$. We see that for all but at most $2m$ many pairs $(a,a')\in \iJ \times \overline{\iJ}$ satisfying $a'-a=c$ we also have that $(a,a'+c'-c)\in \iJ \times \overline{\iJ}$, and conversely, for all but at most $2m$ many pairs $(a,a')\in \iJ \times \overline{\iJ}$ satisfying $a'-a=c'$ we also have that $(a,a'+c-c')\in \iJ \times \overline{\iJ}$. This proves the statement.
\end{proof}

We consider this a `small error', and use the approximation $\sum_{c\in\iIe}\el{\iJ(v_i),c}\approx m\el{\iJ(v_i),\min(\iIe)}$ for each $\iIe\in\cIe$ in order to simplify the definition of $\Core$ below. Our two correction distributions are then defined by the following formulas,
\reminder{$\Corv$}
\reminder{$\Core$}
\begin{align}
 \label{eq:corv}\Prob[\Corv=I]&=\frac{1-\tfrac{m}{\ell}|\{\iJ\in\cJ:\iI\subset \iJ\}|}{|\cJ|}&\mbox{for each $\iI\in\cI$,}\\
 \label{eq:core}\Prob[\Core=\iIe]&=\frac{1-m\sum_{\iJ\in\cJ}\el{\iJ,\min(\iIe)}}{|\cJ|}&\mbox{for each $\iIe\in\cIe$,}
\end{align}
and $\Prob[\Corv=\ast]=1-\sum_{\iI\in\cI}\Prob[\Corv=\iI]$, and $\Prob[\Core=\ast]=1-\sum_{\iIe\in\cIe}\Prob[\Core=\iIe]$. Let us briefly justify that these are really probability distributions, that is, that these formulae are all non-negative. By construction of $\cI$ and $\cJ$, each interval of $\cI$ is in at most $\tfrac{\ell}{m}$ intervals of $\cJ$, so that the expression in~\eqref{eq:corv} is nonnegative. Similarly, by construction $\sum_{\iJ\in\cJ}\el{\iJ,\min(\iIe)}$ is at most $\tfrac{1}{m}$ for each $\iIe\in\cIe$, so that~\eqref{eq:core} is nonnegative. Note that for `most' $\iI$ and $\iIe$, actually~\eqref{eq:corv} and~\eqref{eq:core} evaluate to zero. Now, we have
\begin{equation*}
\Prob[\Corv=\ast]=1-\frac{|\cI|-\tfrac{m}{\ell}|\{(\iI,\iJ)\in\cI\times\cJ:\iI\subset\iJ\}|}{|\cJ|}.
\end{equation*}
Any $\iJ \in \cJ$ contains exactly $\tfrac{\ell}{m}$ intervals $\iI$ from $\cI$. Hence 
\begin{equation}
\label{eq:probcorvast}
 \Prob[\Corv=\ast]=\frac{|\cJ|-|\cI|+|\cJ|}{|\cJ|}
\end{equation}
which is nonnegative (and in fact very close to $1$) by~\eqref{eq:IJsize} and~\eqref{eq:n0}. Finally, we have
\begin{equation*}
\Prob[\Core=\ast]=1-\frac{|\cIe|-m\sum_{\iIe\in\cIe}\sum_{\iJ\in\cJ}\el{\iJ,\min(\iIe)}}{|\cJ|}.
\end{equation*}
Fix $\iJ \in \cJ$ and consider the set of differences $\{km : km = |a - a'|, (a,a')\in\iJ\times\overline{\iJ}\}$. This is a set of size $2(\tfrac{\ell}{m} - 1) + 1$. The largest and the smallest labels in this set can be written as a difference of elements from $\iJ\times\overline{\iJ}$ in exactly $m$ ways each, the second largest and second smallest labels can be written as a difference in $2m$ ways each, and so on. For any $\iIe\in\cIe$, $\min(\iIe)$ is a multiple of $m$, hence
\begin{equation}
\label{eq:probcoreast}
 \Prob[\Core=\ast]=\frac{|\cJ|-|\cIe|+m|\cJ|\sum_{j=-\ell/m}^{\ell/m}\tfrac{\ell-|jm|}{\ell^2} }{|\cJ|}=\frac{2|\cJ|-|\cIe|}{|\cJ|}\,,
\end{equation}
which again by~\eqref{eq:IJsize} and~\eqref{eq:n0} is very close to $1$, and in particular nonnegative.

\medskip

Given sets $\setVert\subset\bbA$ and $\setEdge\subset\bbC$, an interval $I\subset\bbA$ and a label $a$, we say that $a'\in\bbA$ is \reminder{admissible label}\emph{admissible for $a$ and $I$ with respect to $\setVert$ and $\setEdge$} if $a'\in A\cap I$ and $|a'-a|\in C$. We let the set of admissible vertices for $a$ and $I$ with respect to $\setVert$ and $\setEdge$ be $\adm(a,I;\setVert,\setEdge)$. Observe that we have 
\begin{equation}\label{eq:XthreeAdm}
|\Xthree{a,I}(\setVert,\setEdge)|=|\adm(a,I;\setVert,\setEdge)|\;.
\end{equation}

\medskip

We generate a labelling of $V(T)$ by Algorithm~\ref{alg:label}. For $t=1$, at lines~\ref{line:chlabel} and~\ref{line:edgerem} of Algorithm~\ref{alg:label} we define 
\begin{equation}\label{eq:AlgTOne}
\adm(\psi_0(\prt{v_1}),\iJ(v_1);\setVert_1,\setEdge_1):=\iJ(v_1)\quad\mbox{and}\quad
\{|\psi_1(\prt{v_1})-a|\}:=\emptyset\;.
\end{equation}

\definecolor{cadet}{rgb}{0.33, 0.41, 0.47}
\algrenewcomment[2]{\hspace{#1}\textcolor{cadet}{\(\triangleright\) #2}}
\begin{algorithm}[h]
\caption{Labeling of $T$.}
\label{alg:label}
 Let $\psi_0:=\emptyset$, let $\setVert_1:=\bbA$ and let $\setEdge_1:=\bbC$ \;
 \ForEach{$t=1,\dots,n$}{
  Choose\label{line:chlabel} $a\in\adm\big(\psi_{t-1}\big(\prt{v_t}\big),\iJ(v_t);\setVert_t,\setEdge_t\big)$ uniformly at random\Comment{0.45cm}{may fail, \hspace{5cm} \phantom{a} \hspace{10.3cm}  see~\eqref{eq:AlgTOne} for $t=1$}\;
  $\psi_t:=\psi_{t-1}\cup\{v_t\mapsto a\}$\Comment{5.15cm}{enhance the partial labelling}\;
  $\setVert_t^r:=\setVert_t\setminus\{a\}$ \Comment{5.3cm}{remove corresponding vertex label}\;
  $\setEdge_t^r:=\setEdge_t\setminus\{|\psi_t(\prt{v_t})-a|\}$ \label{line:edgerem} \Comment{0.7cm}{remove corresponding edge label, see~\eqref{eq:AlgTOne} for $t=1$}\;
  Sample $x$ from $\Corv$ \label{line:corvjohny}\Comment{5.62cm}{correction on vertex labels}\;
  \If{$x=I\subset\cI$}{
   Choose \label{line:corv} $r^v_{t}\in I\cap \setVert_t^r$ uniformly at random \Comment{4.6cm}{may fail}\;
   $\setVert_{t+1}:=\setVert_t^r\setminus \{r^v_{t}\}$ \;
  }
  \Else{
   $\setVert_{t+1}:=\setVert_t^r$, $r^v_t:=\ast$ \;
  }
  Sample $y$ from $\Core$  \label{line:15}\Comment{5.8cm}{correction on edge labels}\;
  \If{$y=\iIe\in\cIe$}{
   Choose \label{line:core} $r^e_{t}\in \iIe\cap \setEdge_t^r$ uniformly at random      \label{line:17}\Comment{4.5cm}{may fail}\;
   $\setEdge_{t+1}:=\setEdge_t^r\setminus\{r^e_{t}\}$ \;
  }
  \Else{
   $\setEdge_{t+1}:=\setEdge_t^r$, $r^e_t:=\ast$ \;
  }
 }
 \Return $\psi_n$ \;
\end{algorithm}

\subsection{Probabilistic formalities}\label{ssec:ProbablisticForm}

To apply Lemma~\ref{lem:seqhoeff}, as we will want to do, we need a probability space $\Omega$ and a filtration $\cF_0,\cF_1,\dots$. Let $\Omega$ be the set of sequences of length $3n$ over the alphabet $\big[\tilde{n}\big]\cup\{\ast,\mathrm{fail}\}$. We generate a sequence in $\Omega$ from a run of Algorithm~\ref{alg:label} by recording, for each time $t=1,\dots,n$, the vertex label chosen at line~\ref{line:chlabel}, the choice of $r^v_t$ , and the choice of $r^e_t$. In the event that the algorithm fails --- which occurs when it requests to sample a uniform element from an empty set in  lines~\ref{line:chlabel}, \ref{line:corv}, or \ref{line:core} --- we record $\mathrm{fail}$ at the point when the algorithm fails and in all remaining places of the sequence. We obtain a probability measure on $\Omega$ as the probability that running Algorithm~\ref{alg:label} generates a given sequence.

When we use Lemma~\ref{lem:seqhoeff}, we will have random variables $Y_1,\dots,Y_n$ tracking sequential contributions to some parameter. Each $Y_i$ is determined by some initial segment of $\left(\big[\tilde{n}\big]\cup\{\ast,\mathrm{fail}\}\right)^{3n}$, called a \emph{history}, $\hist_i(\omega)$ of $\omega\in\Omega$, we have an estimate for $\sum_i\Exp[Y_i|\hist_{i-1}]$, and the lengths of these initial segments are monotone increasing. Since the lengths of the histories are increasing, they generate in the natural way a filtration on $\Omega$, as required for Lemma~\ref{lem:seqhoeff}.

In the rest of the paper, we will not need the details of this construction of $\Omega$, but simply the observation that conditioning on some history is equivalent to conditioning on the behaviour of Algorithm~\ref{alg:label} up to a given point, and that Lemma~\ref{lem:seqhoeff} applies to random variables of the above type.

\section{Proof of Theorem~\ref{thm:mainred}}
\label{sec:proof}
 \subsection{Technical overview}
  Before starting the proof, we give a brief overview of the structure. Ultimately, all we need to do is show that Algorithm~\ref{alg:label} runs successfully with positive probability. We will show something rather stronger, namely that in fact with high probability, at each time $t$ in the running of Algorithm~\ref{alg:label} the pair $(\setVert_t,\setEdge_t)$ is $\delta_t$-quasirandom. This is a stronger claim because this quasirandomness in particular asserts that the sets from which labels are chosen at lines~\ref{line:chlabel}, \ref{line:corv} and \ref{line:core} are non-empty.
  
  In turn, to prove $\delta_t$-quasirandomness of $(\setVert_t,\setEdge_t)$, we consider separately each $\iIe\in\cIe$ for~\ref{quasi:edged} and each $X\in\cX$ for~\ref{quasi:X}. We describe our approach for a given $X\in\cX$; that for $\iIe\in\cIe$ is analogous. We can write $\big|X(\setVert_t,\setEdge_t)\big|$ as $\big|X(\bbA,\bbC)\big|$ minus the (random) change caused by the (random) choices of $\psi_1(v_1)$, $r_1^v$, $r_1^e$, $\psi_2(v_2)$, and so on up to $r_{t-1}^e$. Thus what we want to do is estimate the sum of a collection of random variables. These random variables are sequentially dependent, so that we can use Lemma~\ref{lem:seqhoeff} to provide such an estimate. We will see that the probability bounds coming from Lemma~\ref{lem:seqhoeff} are strong enough to simply use the union bound over all choices of $X$ and $t$, completing the proof.
  
  The difficulty in this programme is that in order to apply Lemma~\ref{lem:seqhoeff} we need estimates for the expected changes at each step, conditioned on the history up to that step. In order to obtain these estimates, we need to know that $(\setVert_i,\setEdge_i)$ is $\delta_i$-quasirandom for earlier $i$. This may seem suspiciously circular: but it is not. To see this, consider the first time $t$ at which quasirandomness fails. This is the first time at which some sum of changes deviates excessively from its expected value. The probability of this event is bounded by Lemma~\ref{lem:seqhoeff} in terms of the sum of the conditional expectations of changes, and those conditional expectations are calculated assuming $\delta_i$-quasirandomness of $(\setVert_i,\setEdge_i)$ for some values $i<t$, in other words for times $i$ when, because $i<t$ and $t$ is the first time at which quasirandomness fails, we do have this quasirandomness. Lemma~\ref{lem:seqhoeff} then tells us that the event of quasirandomness first failing at time $t$ is unlikely, and sufficiently unlikely that taking a union bound over $t$ we conclude that quasirandomness failing at any time is unlikely.
  
  Let us now discuss how we obtain these sums of conditional expectations. The \emph{removal term} change caused by $r^v_i$ depends only on the sets $\setVert^r_i$ and $\setEdge^r_i$, and similarly the removal term change caused by $r^e_i$ depends only on $\setVert_{i+1}$ and $\setEdge^r_i$. It is thus quite easy to estimate the sums of conditional expectations of these changes, which we do in Claim~\ref{cl:p}. It is rather harder to estimate the change caused by $\psi_i(v_i)$, which in addition to the sets $\setVert_i$ and $\setEdge_i$ depends also on the choice of $\prt{v_i}$, which in turn depends on earlier labellings, and so on; ultimately there is some dependence on the entire labelling. Analysing this seems at first hopeless. But in fact there is only significant dependence on $\prt{v_i}$; we will show that, assuming quasirandomness, any choice of $\prt{{(\prt{v_i})}}$ leads to the conditional expectation of change when labelling $v_i$ being approximately a quantity, $p_{X,i}$, which we call the \emph{crude estimate}. This quantity $p_{X,i}$ does not depend at all on the labelling process, and thus we can quite easily estimate the sum over $i$ of the $p_{X,i}$'s. Putting this estimate together with the estimated sum of removal terms, which we do in Claim~\ref{cl:p}, yields the `correct' value of $\big|X(\bbA,\bbC)\big|-\big|X(\setVert_t,\setEdge_t)\big|$. In other words, it is enough to show that the sum of crude estimates corresponds to the actual changes caused by labelling.
  
  This is still not an easy task. We perform it in two steps. First, we argue in Claim~\ref{cl:q} that $p_{X,i}$ is approximately the expectation of change caused by labelling $v_i$, conditioned on the history up to immediately before labelling $\prt{v_i}$. We define a \emph{fine estimate} $q_{X,i}$ which corresponds to the  expectation of change caused by labelling $v_i$, conditioned on the history up to immediately after labelling $\prt{v_i}$. An application of Lemma~\ref{lem:seqhoeff} then tells us that with high probability the sum of the $p_{X,i}$'s is approximately the sum of the $q_{X,i}$'s. Now, $q_{X,i}$ is still not the conditional expectation we would like to find: some vertices may be labelled in between labelling $\prt{v_i}$ and labelling $v_i$, and these labellings, together with removal terms in the same interval, cause $q_{X,i}$ and the expectation of change caused by labelling $v_i$, conditioned on the history up to immediately before labelling $v_i$, to be different. But provided there are only few such intervening vertices, the difference is small. Our choice of order, using~\ref{pre:int}, guarantees that for most $i$ there are indeed few such intervening vertices, and we conclude that (deterministically) the sum of the $q_{X,i}$ is close to the sum over $i$ of the expectation of change caused by labelling $v_i$, conditioned on the history up to immediately before labelling $v_i$. This last sum is what we need in order to apply Lemma~\ref{lem:seqhoeff} to estimate the sum of the actual changes caused by labelling the $v_i$, which completes the proof.
  
  In total, then, since the $p_{X,i}$ are quantities independent of the labelling process we do not need to assume anything to estimate their sum in Claim~\ref{cl:p}. To show that their sum approximates the sum of the $q_{X,i}$'s with high probability, which we do in Claim~\ref{cl:q}, and to show that the sum of the $q_{X,i}$'s is with high probability close to the sum of the actual changes, we need to assume quasirandomness before the time when we label $v_i$. As this is before time $i$, as discussed this assumption is valid. 
  
 \subsection{Proof of Theorem~\ref{thm:mainred}}

 Given $\gamma>0$, let constants and sets be as defined in Setup~\ref{setup}. Given an $n$-vertex tree $T$ with $\Delta(T)\le\frac{\Pc n}{\log n}$, Lemma~\ref{lem:setup} produces an edge set  $\mathcal{R}$, an ordering $V(T)=\{v_1,\ldots,v_n\}$ and intervals $\iJ(v_i)$.
 
 In order to apply Lemma~\ref{lem:seqhoeff}, we will twice need to use the following upper bound on $\sum_{v\in V(T)}\deg(v)^2$.
 \begin{equation}\label{eq:treedegsq}
  \sum_{v\in V(T)}\deg(v)^2\le\Delta(T)\cdot \sum_{v\in V(T)}\deg(v)\le \frac{\Pc n}{\log n}
\cdot 2e(T)\le \frac{2\Pc n^2}{\log n}\,.
 \end{equation}

 We run Algorithm~\ref{alg:label}. We say the algorithm \emph{fails} if at any time a step is not possible: in other words, if the sets from which we should choose uniformly $\psi_t(v_t)$, $r^v_t$, or $r^e_t$, are empty. Then Theorem~\ref{thm:mainred} holds if with positive probability Algorithm~\ref{alg:label} does not fail. We will show that with high probability Algorithm~\ref{alg:label} maintains the property that $(\setVert_t,\setEdge_t)$ is $\delta_t$-quasirandom for each $1\le t\le n$.
 
 For each $1\le \tau\le n$, let us define the following two events:
 \begin{align}
 \nonumber
 \mathcal{W}_{<\tau}&:=\left\{\mbox{Algorithm~\ref{alg:label} has not failed before time $\tau$}\right\}\;,\\
 \label{eq:sizeAC}
 \mathcal{U}_{\tau}&:=\left\{\mbox{we have $|\setVert_\tau|,|\setEdge_\tau|=\tilde{n}-\tau\pm 10\ell$}\right\}
 \end{align}
 
 \begin{claim}\label{cl:sizeAC}
We have for the probability of the `bad event' $\bigcup_{\tau=1}^n (\mathcal{W}_{<\tau}\setminus \mathcal{U}_{\tau})$ we have
$$\Prob\left[\bigcup_{\tau=1}^n (\mathcal{W}_{<\tau}\setminus \mathcal{U}_{\tau})\right]\le n^{-1}\;.$$

 \end{claim}
 Rephrasing the claim, with probability at least $1-n^{-1}$, at each time $1\le \tau\le n$, it holds that Algorithm~\ref{alg:label} has failed before time $\tau$ or we have~\eqref{eq:sizeAC}.
 \begin{claimproof}[Proof of Claim~\ref{cl:sizeAC}] 
Observe that, unless the algorithm fails, in each step after the first (in which no edge label is given to any edge) one vertex label and one edge label is used in the labelling, so $|\setVert_\tau|\le\tilde{n}-\tau+1$ and $|\setEdge_\tau|\le|\bbC|-\tau+2=\tilde{n}-\tau+1$, as needed for the upper-bound.

Let us now turn to the lower-bound. For $t\in[n]$, let $U_t$ be defined as follows:
 \begin{enumerate}[label=(\alph*)]
 	\item\label{jh} if Algorithm~\ref{alg:label} has not failed until step $t$, let $U_t$ be the indicator that $*$ was not sampled on Line~\ref{line:corvjohny} (in step $t$),
 	\item if Algorithm~\ref{alg:label} has failed before step $t$, let $U_t$ be a Bernoulli random variable with success probability $\frac{2(\ell-m)}{\tilde{n}-2(\ell-m)}$ (and independent of all other random choices).
 \end{enumerate}
 For $t\in[n]$, let $W_t$ be defined as follows in the same way, except that in case~\ref{jh}, we use the indicator that $*$ was not sampled on Line~\ref{line:15}.
 
 By~\eqref{eq:IJsize},~\eqref{eq:probcorvast} and~\eqref{eq:probcoreast}, we have
 \[\Prob[\Corv\neq\ast]=\Prob[\Core\neq\ast]=\frac{2(\ell-m)}{\tilde{n}-2(\ell-m)}\,.\]
Therefore, $U_t$'s and $W_t$'s are independent Bernoulli random variable with success probability $\frac{2(\ell-m)}{\tilde{n}-2(\ell-m)}<\frac{2\ell}{n}$. By Theorem~\ref{thm:hoeff}, the probability that $\sum_{t=1}^n U_t>10\ell$ or $\sum_{t=1}^n W_t>10\ell$ is at most $n^{-1}$. That is, with probability at least $1-n^{-1}$ we have that Algorithm~\ref{alg:label} failed or $|\setVert_n|,|\setEdge_n|\ge\tilde{n}-n-10\ell$. Observe that in the case of this good event, we also get $|\setVert_\tau|,|\setEdge_\tau|\ge\tilde{n}-\tau-10\ell$ for each $\tau\le n$, no matter at which times the potential non-$\ast$ samples were sampled.
 \end{claimproof}

The next claim tells us that Algorithm~\ref{alg:label} does not fail at line~\ref{line:chlabel} (Claim~\ref{cl:sizespositive}\ref{en:sizeposA}), line~\ref{line:corv} (Claim~\ref{cl:sizespositive}\ref{en:sizeposB}), nor at line~\ref{line:core} (Claim~\ref{cl:sizespositive}\ref{en:sizeposC}).
\begin{claim}\label{cl:sizespositive}
	 Suppose that $(\setVert_t,\setEdge_t)$ is $\delta_t$-quasirandom and~\eqref{eq:sizeAC} holds. Then we have
\begin{enumerate}[label=(\alph*)]
	\item\label{en:sizeposA} $\big|\Xthree{a,I}(\setVert_t,\setEdge_t)\big|>0$,
	\item\label{en:sizeposB} $\big|\Xone{I}(\setVert_t,\setEdge_t)\big|>0$, and
	\item\label{en:sizeposC} $|\iIe\cap \setEdge_t^r|>0$.
\end{enumerate}
\end{claim}
\begin{claimproof}
 Let us first prove~\ref{en:sizeposA}. By choice of $\delta_t$, for each $a\in\bbA$ and $\iI\in\cI$, using~\ref{quasi:X} we have 
\begin{align*}
\big|\Xthree{a,I}(\setVert_t,\setEdge_t)\big|&
=\big|\Xthree{a,I}(\bbA,\bbC)\big|\cdot \left(\frac{|\setVert_t|}{\tilde{n}}\right)^2\pm\delta_t m\\
\JUSTIFY{by Lemma~\ref{lem:expressXcount}\ref{en:fullX3cI} and \eqref{eq:sizeAC}}
&\ge m\cdot \left(\frac{\tilde{n}-n-10\ell}{\tilde{n}}\right)^2-\delta_t m\\
\JUSTIFY{Setup~\ref{setup}, and $\delta_t\le \delta_n=\mu^{1/\mu}$}
&\ge \left(\gamma -10\delta_0^2-\mu^{1/\mu}\right)m>0\;.
 \end{align*}
 
Mutatis mutandis we obtain~\ref{en:sizeposB}.
 
Last, let us turn to~\ref{en:sizeposC}. We have
 $$|\iIe\cap \setEdge_t^r|= |\iIe\cap \setEdge_t|\pm 1\eqBy{\ref{quasi:edged}}
 m\cdot\frac{|\setVert_t|}{\tilde{n}}\pm\delta_t m \pm 1
 \geByRef{eq:sizeAC}m\cdot\frac{\tilde{n}-n-10\ell}{\tilde{n}}-\delta_t m-1>0\;.
 $$ 
 \end{claimproof}
Thus, in order to prove Theorem~\ref{thm:mainred}, it is enough to show that with high probability $(\setVert_t,\setEdge_t)$ is $\delta_t$-quasirandom for each $1\le t\le n$. We now embark upon proving this.
 
 Since the vertex labels $r^v_t$ and $r^e_t$ are chosen uniformly at random within intervals of respectively $\cI$ and $\cIe$, it is quite easy to analyse their effect on~\ref{quasi:edged} and~\ref{quasi:X}. It is rather harder to analyse the effect of the edge and vertex labels used at step $t$, since these are not chosen uniformly. However, the idea one should have in mind is that this choice is `close to uniform' in a sense we will make precise later, and thus it is useful to write down `crude estimates' for the effect of the vertex and edge labels used at step $t$ in the labelling which pretends these choices are really uniform. Specifically, the following estimates correspond (more or less) to the expected effect if $\phi_t(v_t)$ were chosen uniformly from the unused vertex labels in $\iJ(v_t)$, if the edge label $\big|\phi_t(v_t)-\phi_t(\prt{v_t})\big|$ were chosen to be $c\in \setEdge_t$ with probability  proportional to $\el{\iJ(v_t),c}$, independently, and if~\ref{quasi:edged} and~\ref{quasi:X} held with zero error at time $t$. Of course all these assumptions are false, but we will see that `on average' they hold, which is enough for our proof.
 
 For $\iIe\in\cIe$ and $1\le t\le n$, we define
   \reminder{$p_{\iIe,t}$}
   \reminder{$r_{\iIe,t}$}
 \begin{equation}\label{eq:edged:crude}\begin{split}
  \text{\emph{the crude estimate}}\qquad p_{\iIe,t}&:= m \cdot  \el{\iJ(v_t),\min(\iIe)} \quad\text{and}
  \\
  \text{\emph{the removal term}}\qquad r_{\iIe,t}&:=\ind_{r^e_{t}\in\iIe}\,.
 \end{split}\end{equation}
 The crude estimate is an idealised version of the expected number of edge labels in $\iIe$ that are used at time $t$ (equivalently, to the probability that at time $t$ we use an edge label in the interval $\iIe$). The removal term is the indicator of the event that at time $t$ we remove an edge label from $\iIe$. For $X\in\cX$ we define similar terms, with the same intent. Again, the crude estimate $p_{X,t}$ is an estimate for the expected change $\big|X(\setVert_t,\setEdge_t)\big|-\big|X(\setVert^r_t,\setEdge^r_t)\big|$, and the removal term $r_{X,t}$ is the actual change $\big|X(\setVert^r_t,\setEdge^r_t)\big|-\big|X(\setVert_{t+1},\setEdge_{t+1})\big|$. For the latter, recall that $\setVert_t^r$ and $\setEdge_t^r$ are the available vertex and edge labels, respectively, at time $t$ after removing the vertex label and edge label used in labelling $T$. So,
   \reminder{$p_{X,t}$}
   \reminder{$r_{X,t}$}
 \begin{equation}\label{eq:struct:crude}\begin{split}
  p_{X,t}&:=\tfrac{|X(\bbA,\bbC)|(\tilde{n}-t)^{\free(X)-1}}{\tilde{n}^{\free(X)-1}}\Big(\sum_{I\in\freev(X)}\tfrac{\ind_{I\subset\iJ(v_t)}}{\ell}+\sum_{\ed\in\freee(X)}\el{\iJ(v_t),\diff(\ed;X)}  \Big),   \text{and}\\
\quad r_{X,t}&:=\big|\big\{X'\in X(\setVert_t^r,\setEdge_t^r):r^v_t\in\chosenv{X'}{X}\text{ or }r^e_t\in\chosene{X'}{X}\big\}\big| \,.
 \end{split}\end{equation}
 
 Note that the crude estimates $p_{\iIe,t}$ and $p_{X,t}$ are determined before the algorithm starts, and so are their corresponding partial sums $\sum_{i=1}^tp_{\iIe,i}$ and $\sum_{i=1}^tp_{X,i}$. By contrast, the partial sums of the removal terms, $\sum_{i=1}^tr_{\iIe,i}$ and $\sum_{i=1}^tr_{X,i}$, are sums of random variables which in principle depend upon all of the random choices we make throughout the labelling. However, if we assume that $(\setVert_i,\setEdge_i)$ is $\delta_i$-quasirandom for each $1\le i\le t-1$ then we can obtain good bounds on these partial sums which hold with high probability by considering only the choice of the $r^e_i$ and $r^v_i$.
 
 We estimate the partial sums $\sum_{i=1}^tr_{\iIe,i}$ and $\sum_{i=1}^tr_{X,i}$ together. The reason is that eventually we will be able to show that (for example) $\sum_{i=1}^tp_{\iIe,i}$ is with high probability a good estimate for the number of edge labels in $\iIe$ used in the labelling up to time $t$, and it follows that $\sum_{i=1}^t(p_{\iIe,i}+r_{\iIe,i})$ is a good estimate for $|\bbC\cap \iIe|-|\setEdge_t\cap\iIe|$, which is what we want to know in order to verify~\ref{quasi:edged}. Recall that in the introduction we mentioned that our proof can be seen as an application of the Differential Equations Method. This claim is where we (implicitly) verify that the crude estimates we chose actually correspond to solutions to some differential equations: one should understand the right hand sides of~\eqref{eq:sumpedged} and~\eqref{eq:sumpX} as (what we expect for) the differences $|\bbC\cap \iIe|-|\setEdge_t\cap\iIe|$ and $|X(\bbA,\bbC)|-|X(\setVert_t,\setEdge_t)|$, respectively.
 
 \begin{claim}\label{cl:p}
  With probability at least $1-2n^{-1}$, for each $\iIe\in\cIe$, each $X\in\cX$, and each $1\le t\le n$, if $(\setVert_i,\setEdge_i)$ is $\delta_i$-quasirandom for each $1\le i<t$ we have
  \begin{align}
   \label{eq:sumpedged}\sum_{i=1}^t(p_{\iIe,i}+r_{\iIe,i})&=\frac{t}{|\cJ|}\pm\tfrac14\delta_t m\;\mbox{, and}\\
   \label{eq:sumpX}\sum_{i=1}^t(p_{X,i}+r_{X,i})&=|X(\bbA,\bbC)|\frac{\tilde{n}^{\free(X)}-(\tilde{n}-t)^{\free(X)}}{m|\cJ|\tilde{n}^{\free(X)-1}}\pm\tfrac14\delta_t m\,.
 \end{align}
 \end{claim}
 \begin{claimproof}[Proof of Claim~\ref{cl:p}]
We have
\begin{align}
\nonumber
  \sum_{i=1}^t p_{\iIe,i}
  &=
  \sum_{i=1}^t \sum_{\iJ\in\cJ} \ind_{\iJ(v_i)=\iJ} \cdot p_{\iIe,i}\\
\nonumber  
\JUSTIFY{by~\eqref{eq:edged:crude}}  &=
    \sum_{i=1}^t \sum_{\iJ\in\cJ} \ind_{\iJ(v_i)=\iJ}\cdot m \el{\iJ,\min(\iIe)}\\
\label{eq:HarryPotter1}    
  \JUSTIFY{by~\ref{pre:chJ}, with $S=[t]$}&=
  \big(\tfrac{t}{|\cJ|}\pm\delta^2 n\big)\sum_{\iJ\in\cJ}m\el{\iJ,\min(\iIe)}\,.
\end{align}
Let us now turn to the quantity $\sum_{i=1}^tr_{\iIe,i}$. We have
\begin{align*}
\Exp\left[\sum_{i=1}^t r_{\iIe,i}\right]
&\eqByRef{eq:edged:crude}
\Exp\left[\sum_{i=1}^t \ind_{r^e_{i}\in\iIe}\right]\\
\JUSTIFY{by the way $r^e_{i}$ is chosen on lines~\ref{line:15} and~\ref{line:17} of Algorithm~\ref{alg:label}}&=
t\cdot \Prob[\Core=\iIe]\\
\JUSTIFY{by~\eqref{eq:core}}&=t\cdot \frac{1-\sum_{\iJ\in\cJ}m\el{\iJ,\min(\iIe)}}{|\cJ|}\,.
\end{align*}
The events $\{r^e_{i}\in\iIe\}_{i=1}^t$ are independent, and thus Theorem~\ref{thm:hoeff} gives us that 
$$\sum_{i=1}^t r_{\iIe,i}=t\cdot \frac{1-\sum_{\iJ\in\cJ}m\el{\iJ,\min(\iIe)}}{|\cJ|}\pm \delta^2n$$
with probability at least $1-n^{-10}$.
Putting this together with~\eqref{eq:HarryPotter1}, we get that~\eqref{eq:sumpedged} holds for that $t$ and that $\iIe$ with probability at least $1-n^{-10}$. Taking the union bound over all choices of $t$ and $\iIe\in\cIe$, we see that with probability at least $1-n^{-2}$,~\eqref{eq:sumpedged} holds for all $t$ and $\iIe$ as desired.
  
  For~\eqref{eq:sumpX} we need to be a little more careful: the quantity $p_{X,i}$ depends on $i$ as well as $\iJ(v_i)$, and (with similar effect) the quantity $r_{X,i}$ depends on $|X(\setVert_i^r,\setEdge_i^r)|$ and $|\setVert_i^r\cap \iI|$ and $|\setEdge_i^r\cap \iIe|$ for each $\iI\in\cI$ and $\iIe\in\cIe$ as well as the outcomes of $\Corv$ and $\Core$. We divide the interval $[t]$ into intervals \reminder{$S_j$}$S_1,\dots,S_{\lceil t/\delta n\rceil}$, all except possibly the last consisting of $\delta n$ elements. Note that by assumption $\delta n$ is an integer which divides $n$. The point of doing this is that any time $i\in S_j$, $|X(\setVert_i^r,\setEdge_i^r)|$, $|\setVert_i^r\cap \iI|$ and $|\setEdge_i^r\cap \iIe|$ are up to a small error constant on any $S_j$.
  
  For $i\in S_j$, since $\tilde{n}-j\delta n\ge\gamma n$, we have $$\tilde{n}-i\le \tilde{n}-j\delta n+\delta n\le (\tilde{n}-j\delta n)\big(1+\tfrac{\delta}{\gamma}\big)\;.$$ Thus, for any $1\le s\le 3$,
  \[
   (\tilde{n}-i)^{s}\le(\tilde{n}-j\delta n)^s\big(1+\tfrac{\delta}{\gamma}\big)^s\le \big(1+8\delta\gamma^{-1}\big)(\tilde{n}-j\delta n)^s\,,
  \]
  and hence $(\tilde{n}-i)^s=\big(1\pm 8\delta\gamma^{-1}\big)(\tilde{n}-j\delta n)^s$. Using this and~\ref{pre:chJ}, for each $X\in\cX$ we have
  \begin{align}
  \begin{split}
  \label{eq:p:pXintro}
   \sum_{i\in S_j}p_{X,i}&=\big(1\pm \tfrac{8\delta}{\gamma}\big)\frac{|X(\bbA,\bbC)|(\tilde{n}-j\delta n)^{\free(X)-1}}{\tilde{n}^{\free(X)-1}}\sum_{I\in\freev(X)}\big(\tfrac{|S_j|}{|\cJ|}\pm\delta^2 n\big)\frac{|\{\iJ\in\cJ:I\subset\iJ\}|}{\ell}\\
   &+\big(1\pm \tfrac{8\delta}{\gamma}\big)\frac{|X(\bbA,\bbC)|(\tilde{n}-j\delta n)^{\free(X)-1}}{\tilde{n}^{\free(X)-1}}\sum_{\ed\in\freee(X)}\sum_{\iJ\in\cJ}\big(\tfrac{|S_j|}{|\cJ|}\pm\delta^2 n\big)\el{\iJ,\diff(\ed;X)}
   \end{split}\\[4mm] 
\begin{split}   \label{eq:p:pX}
   &=\frac{|X(\bbA,\bbC)|(\tilde{n}-j\delta n)^{\free(X)-1}}{\tilde{n}^{\free(X)-1}}\sum_{I\in\freev(X)}\frac{|S_j|\cdot|\{\iJ\in\cJ:I\subset\iJ\}|}{|\cJ|\ell}\\
   &+\frac{|X(\bbA,\bbC)|(\tilde{n}-j\delta n)^{\free(X)-1}}{\tilde{n}^{\free(X)-1}}\sum_{\ed\in\freee(X)}\sum_{\iJ\in\cJ}\tfrac{|S_j|}{|\cJ|}\el{\iJ,\diff(\ed;X)}\pm 100\delta^2n\gamma^{-1}\,.
  \end{split}\end{align}
  Let us hint where the value of the final error term $\pm 100\delta^2n\gamma^{-1}$ in~\eqref{eq:p:pX} comes from. There are two error terms in~\eqref{eq:p:pXintro}. To bound the error introduced by the term $\pm \tfrac{8\delta}{\gamma}$, we use that 
  \begin{align*}
&\frac{|X(\bbA,\bbC)|(\tilde{n}-j\delta n)^{\free(X)-1}}{\tilde{n}^{\free(X)-1}}\leBy{F\ref{fact:zada}\ref{en:ZadaB}} m\;, \\
& \sum_{I\in\freev(X)}\big(\tfrac{|S_j|}{|\cJ|}\pm\delta^2 n\big)\cdot\frac{|\{\iJ\in\cJ:I\subset\iJ\}|}{\ell}
 \le 2\big(\tfrac{\delta n}{2\frac{n}m}+\delta^2 n\big)\cdot\frac{\delta_0^{-2}}{\ell}\le 5 \delta\;\text{, and}\\
&\sum_{\ed\in\freee(X)}\sum_{\iJ\in\cJ}\big(\tfrac{|S_j|}{|\cJ|}\pm\delta^2 n\big)\cdot \el{\iJ,\diff(\ed;X)}\leBy{F\ref{fact:alexandria}\ref{en:Ocasio}} 2\cdot \nicefrac{2}{\delta_0^2}\cdot \big(\tfrac{\delta n}{2\frac{n}m}+\delta^2 n\big)\cdot \frac{1}{\ell}\le 5\delta\;.
  \end{align*}
The error coming from the term $\pm \delta^2n$ can be bounded similarly.
  
   We now estimate $\sum_{i\in S_j}r_{X,i}$. We will use Lemma~\ref{lem:seqhoeff} to do this. To that end, for each $i=0,\dots,n-1$, let $\hist_i$ be the history up to and including the choice of $\psi_{i+1}(v_{i+1})$, and let $\hist_n$ be the complete history. Hence, the difference between $\hist_{n-1}$ and $\hist_{n}$ is only in the information about the choice of $r^v_n$ and $r^e_n$. Let $\cE_t$ be the event that $(\setVert_i,\setEdge_i)$ is $\delta_i$-quasirandom for each $1\le i<t$ from which we subtract the event $\bigcup_{\tau=1}^n (\mathcal{W}_{<\tau}\setminus \mathcal{U}_{\tau})$ from Claim~\ref{cl:sizeAC}.\footnote{We emphasize that the event $\bigcup_{\tau=1}^n (\mathcal{W}_{<\tau}\setminus \mathcal{U}_{\tau})$ involves conditions even on times $\tau>t$.} Suppose now that $\cE_t$ occurs. That is in the calculations below, we shall work with an arbitrary conditional subspace $\hist_i$, for some $i<t$, but only with such that $\hist_i\cap\cE_t$ has positive probability.
   
   Since $(\setVert_i^r,\setEdge_i^r)$ differs by one vertex and one edge label\footnote{with the only exception $i=1$ when we have $\setEdge_i^r=\setEdge_i$} from $(\setVert_i,\setEdge_i)$, and for any given $X\in\cX$ these two labels meet at most three $X'\in X(\setVert_i,\setEdge_i)$, we have $\big|X(\setVert_i^r,\setEdge_i^r)\big|=\big|X(\setVert_i,\setEdge_i)\big|\pm 3$. Furthermore, $(\setVert_{i+1},\setEdge_i^r)$ differs by at most one vertex label from $(\setVert_i^r,\setEdge_i^r)$, and this vertex label meets at most one $X'\in X(\setVert_i^r,\setEdge_i^r)$. Thus, using~\ref{quasi:X} to estimate $\big|X(\setVert_i,\setEdge_i)\big|$, for all $i\in S_j$ we have
  \begin{equation}\label{eq:p:sizeX}\begin{split}
   \big|X(\setVert_{i+1},\setEdge_i^r)\big|,\big|X(\setVert_i^r,\setEdge_i^r)\big|&=\frac{|X(\bbA,\bbC)| \cdot|\setVert_i|^{\free(X)}}{\tilde{n}^{\free(X)}}\pm\delta_i m \pm 4\\
   &\eqByRef{eq:sizeAC}\frac{|X(\bbA,\bbC)|(\tilde{n}-i\pm 10\ell)^{\free(X)}}{\tilde{n}^{\free(X)}}\pm\delta_i m\pm 4\\
   &=\frac{|X(\bbA,\bbC)|(\tilde{n}-j\delta n)^{\free(X)}}{\tilde{n}^{\free(X)}}\pm 2\delta_i m\;.
  \end{split}\end{equation}
  
  Given $\iI\in\cI$, by~\ref{quasi:X} with the structure $\Xone{I}$, we have
  \begin{equation}\label{eq:AcapI}
  |\setVert_i^r\cap I|=m\cdot \frac{\tilde{n}-i\pm 10\ell}{\tilde{n}}\pm\delta_i m=m\cdot \frac{\tilde{n}-j\delta n}{\tilde{n}}\pm 2\delta_i m\,.
  \end{equation}
  Using this together with~\eqref{eq:p:sizeX}, since $r_i^v$ is chosen uniformly in $\setVert_i^r\cap\iI$ for an interval $I\in\cI$ drawn from $\Corv$, we have
  \begin{align*}
    &\Exp\Big[\big|\big\{X'\in X(\setVert_i^r,\setEdge_i^r):r_i^v\in\chosenv{X'}{X}\big\}\big|\Big|\hist_{i-1}\Big]\\
   \JUSTIFY{by \eqref{eq:p:sizeX} and \eqref{eq:AcapI}} =&\sum_{I\in\freev(X)}\Prob[\Corv=I]\frac{\frac{|X(\bbA,\bbC)|(\tilde{n}-j\delta n)^{\free(X)}}{\tilde{n}^{\free(X)}}\pm 2\delta_i m}{m\tfrac{\tilde{n}-j\delta n}{\tilde{n}}\pm 2\delta_i m}\\
   \JUSTIFY{{by \eqref{eq:corv}}}=
   &\Big(\sum_{I\in\freev(X)}\tfrac{1-\tfrac{m}{\ell}|\{\iJ\in\cJ:I\subset\iJ\}|}{|\cJ|}\Big)\cdot\tfrac{|X(\bbA,\bbC)|(\tilde{n}-j\delta n)^{\free(X)-1}}{m\tilde{n}^{\free(X)-1}}\pm \tfrac{10\delta_i}{\gamma|\cJ|}\,.
  \end{align*}
  
  Observe that, since the free vertex labels of $X\in\cX$ are distinct members of $\cI$, they are disjoint and hence 
  \begin{equation}\label{eq:UpperBoundX}
 \text{ any given vertex label is in $\chosenv{X'}{X}$ for at most one $X'\in X(\bbA,\bbC)$. }
  \end{equation}
  
  Let us now fix an interval $S_j$. We apply Lemma~\ref{lem:seqhoeff}, with $|S_j|$ many random variables $\big|\big\{X'\in X(\setVert_i^r,\setEdge_i^r):r_i^v\in\chosenv{X'}{X}\big\}\big|$ for $i\in S_j$ and with the event $\cE_t$. Observe that these random variables are upper-bounded by~1 by~\eqref{eq:UpperBoundX}. Then Lemma~\ref{lem:seqhoeff} states that if $\cE_t$ occurs, then with probability at least $1-2\exp\big(-\tfrac{2\delta^4n^2}{n}\big)\ge 1-n^{-10}$, we have
  \begin{equation}\label{eq:p:rvX}\begin{split}
   &\sum_{i\in S_j}\big|\big\{X'\in X(\setVert_i^r,\setEdge_i^r):r_i^v\in\chosenv{X'}{X}\big\}\big|\\
   =&|S_j|\Big(\sum_{I\in\freev(X)}\tfrac{1-\tfrac{m}{\ell}|\{\iJ\in\cJ:I\subset\iJ\}|}{|\cJ|}\Big)\cdot\tfrac{|X(\bbA,\bbC)|(\tilde{n}-j\delta n)^{\free(X)-1}}{m\tilde{n}^{\free(X)-1}}\pm\tfrac{10\delta n\delta_{j\delta n}}{\gamma|\cJ|}\pm\delta^2 n\\
   =&|S_j|\Big(\sum_{I\in\freev(X)}\tfrac{1-\tfrac{m}{\ell}|\{\iJ\in\cJ:I\subset\iJ\}|}{|\cJ|}\Big)\cdot\tfrac{|X(\bbA,\bbC)|(\tilde{n}-j\delta n)^{\free(X)-1}}{m\tilde{n}^{\free(X)-1}}\pm \tfrac{20\delta n\delta_{j\delta n}}{\gamma|\cJ|}\,.
  \end{split}\end{equation}
  
  We now turn to estimating the effects of the $r_i^e$. Let us fix an arbitrary history $\hist_{i-1}\subset \cE_t$ which leads to a given set $\setEdge_i^r$. By~\ref{quasi:edged}, for each $\iIe\in\cIe$ we have $|\setEdge_i^r\cap\iIe|=m\tfrac{\tilde{n}-j\delta n}{\tilde{n}}\pm 2\delta_i m$. Given any set $L\subset \setEdge_i^r$ of edge labels with $\max(L)-\min (L)\le m$, the set $L$ is contained in two consecutive intervals of $\cIe$. Let these be $\iIe^{(1)}$ and $\iIe^{(2)}$, and let $L_1$ and $L_2$ be the corresponding subsets of $L$. Suppose that $c\in\iIe^{(1)}\cup\iIe^{(2)}$. Then Fact~\ref{fact:alexandria}\ref{en:Ocasio} tells us that $\el{\iJ,\min(\iIe^{(1)})}=\el{\iJ,c}=\el{\iJ,\min(\iIe^{(2)})}=0$ for all but at most $2\tfrac{\ell}{m}$ choices of $\iJ\in\cJ$.  By Fact~\ref{fact:alexandria}\ref{en:Cortez} the three quantities never differ by more than $\tfrac{2m}{\ell^2}$. Thus we have
  \[\sum_{\iJ\in\cJ}m\el{\iJ,\min(\iIe^{(g)})}=\sum_{\iJ\in\cJ}m\el{\iJ,c}\pm8\tfrac{m}{\ell}\]
  for each $g=1,2$. Using this, we have for an arbitrary history $\hist_{i-1}\subset \cE_t$ which leads to the given set $\setEdge_i^r$, that
  \begin{align*}
   \Prob[r_i^e\in L|\hist_{i-1}]&=\frac{1-\sum_{\iJ\in\cJ}m\el{\iJ,\min(\iIe^{(1)})}}{|\cJ|}\cdot\frac{|L_1|}{m\tfrac{\tilde{n}-j\delta n}{\tilde{n}}\pm 2\delta_i m}\\
   &+\frac{1-\sum_{\iJ\in\cJ}m\el{\iJ,\min(\iIe^{(2)})}}{|\cJ|}\cdot\frac{|L_2|}{m\tfrac{\tilde{n}-j\delta n}{\tilde{n}}\pm 2\delta_i m}\\
   &=\frac{1\pm8\tfrac{m}{\ell}-\sum_{\iJ\in\cJ}m\el{\iJ,c}}{|\cJ|}\cdot\frac{|L|}{m\tfrac{\tilde{n}-j\delta n}{\tilde{n}}\pm 2\delta_i m}\\
   &=\frac{\big(1-\sum_{\iJ\in\cJ}m\el{\iJ,c}\big)|L|}{|\cJ|m\tfrac{\tilde{n}-j\delta n}{\tilde{n}}}\big(1\pm 10\gamma^{-1}\delta_i\big)\pm\frac{16|L|}{\ell|\cJ|\gamma}\,,
  \end{align*}
  where $c\in\iIe^{(1)}\cup\iIe^{(2)}$ is arbitrary.
  Recall that $\setVert_{i+1}=\setVert_i^r\setminus\{r_i^v\}$. Using the above calculation and~\eqref{eq:p:sizeX}, we obtain that when $i\in S_j$,
  \begin{align*}
  &\Exp\Big[\big|\big\{X'\in X(\setVert_{i+1},\setEdge_i^r):r_i^e\in\chosene{X'}{X}\big\}\big|\Big|\hist_i,r_i^v\Big]\\
   =&\Big(\tfrac{|X(\bbA,\bbC)|(\tilde{n}-j\delta n)^{\free(X)}}{\tilde{n}^{\free(X)}}\pm 2\delta_i m\Big)\sum_{\ed\in\freee(X)}\Big(\tfrac{(1-\sum_{\iJ\in\cJ}m\el{\iJ,\diff(\ed;X)}}{|\cJ|m\tfrac{\tilde{n}-j\delta n}{\tilde{n}}}\big(1\pm 10\gamma^{-1}\delta_i\big)\pm\tfrac{16}{\ell|\cJ|\gamma}\Big)\\
   =&\tfrac{|X(\bbA,\bbC)|(\tilde{n}-j\delta n)^{\free(X)}}{\tilde{n}^{\free(X)}}\sum_{\ed\in\freee(X)}\tfrac{(1-\sum_{\iJ\in\cJ}m\el{\iJ,\diff(\ed;X)}}{|\cJ|m\tfrac{\tilde{n}-j\delta n}{\tilde{n}}}\pm\tfrac{40m}{\ell|\cJ|\gamma}\pm \tfrac{50\delta_i}{\gamma|\cJ|}\\
   =&\tfrac{|X(\bbA,\bbC)|(\tilde{n}-j\delta n)^{\free(X)-1}}{\tilde{n}^{\free(X)-1}}\sum_{\ed\in\freee(X)}\tfrac{(1-\sum_{\iJ\in\cJ}m\el{\iJ,\diff(\ed;X)}}{|\cJ|m}\pm \tfrac{60\delta_i}{\gamma|\cJ|}\,.
  \end{align*}
  Now any given edge label is in at most four $X'\in X(\setVert_{i+1},\setEdge_i^r)$, and $\hist_i,r_i^v$ is a history, so by Lemma~\ref{lem:seqhoeff}, if $\cE_t$ occurs then with probability at least $1-2\exp\big(-\tfrac{2\delta^4n^2}{16n}\big)\ge 1-n^{-10}$ we have
  \begin{equation}\label{eq:p:reX}\begin{split}
   &\sum_{i\in S_j}\big|\big\{X'\in X(\setVert_i^r,\setEdge_i^r):r_i^e\in\chosene{X'}{X}\big\}\big|\\
   &=|S_j|\tfrac{|X(\bbA,\bbC)|(\tilde{n}-j\delta n)^{\free(X)-1}}{\tilde{n}^{\free(X)-1}}\Big(\sum_{\ed\in\freee(X)}\tfrac{(1-\sum_{\iJ\in\cJ}m\el{\iJ,\diff(\ed;X)})}{|\cJ|m}\Big)\pm \tfrac{60\delta n\delta_{j\delta n}}{\gamma|\cJ|}\pm\delta^2 n\\
   &=|S_j|\tfrac{|X(\bbA,\bbC)|(\tilde{n}-j\delta n)^{\free(X)-1}}{\tilde{n}^{\free(X)-1}}\Big(\sum_{\ed\in\freee(X)}\tfrac{(1-\sum_{\iJ\in\cJ}m\el{\iJ,\diff(\ed;X)})}{|\cJ|m}\Big)\pm \tfrac{70\delta n\delta_{j\delta n}}{\gamma|\cJ|}\,.
  \end{split}\end{equation}
  
  Putting together~\eqref{eq:p:pX},~\eqref{eq:p:rvX} and~\eqref{eq:p:reX}, with probability at least $1-n^{-9}$ we have that if $\cE_t$ occurs then 
  \begin{equation}\label{eq:p:Sj}
   \sum_{i\in S_j}(p_{X,i}+r_{X,i})=\frac{|X(\bbA,\bbC)|(\tilde{n}-j\delta n)^{\free(X)-1}}{\tilde{n}^{\free(X)-1}}|S_j|\cdot\frac{\free(X)}{m|\cJ|}\pm\frac{200\delta n\delta_{j\delta n}}{\gamma|\cJ|}\,.
  \end{equation}
  
  By the union bound over all $X\in\cX$, all $1\le t\le n$ and all $j$, we see that with probability at least $1-n^{-2}$, if $\cE_t$ occurs then the equation~\eqref{eq:p:Sj} holds for all $X\in\cX$, all times $t$ and all sets $S_j$. Now one part of $\cE_t$ is the assumption that $(\setVert_i,\setEdge_i)$ is $\delta_i$-quasirandom for each $1\le i<t$, and the other part is the good event of Claim~\ref{cl:sizeAC}. The latter event occurs with probability at least $1-n^{-1}$ by Claim~\ref{cl:sizeAC}, so that with probability at least $1-2n^{-1}$ the following holds. Whenever $t$ is such that $(\setVert_i,\setEdge_i)$ is $\delta_i$-quasirandom for $1\le i<t$, we have~\eqref{eq:p:Sj} for each $X\in\cX$ and $j$.
  
  Suppose now that, for some $t$, we have~\eqref{eq:p:Sj} for each $X\in\cX$ and $j$. To complete the proof of the claim, we need to show that putting together these partial sums on short intervals, we obtain the desired~\eqref{eq:sumpX}. Here we are implicitly verifying that we have a solution to a certain first-order differential equation (which we do not write down as we do not need to know it), and consequently an integral naturally appears.
  
  We have
  \[\int_{x=0}^t(\tilde{n}-x-\delta n)^{\free(X)-1}dx\le\sum_{j=1}^{\lceil\frac{t}{\delta n}\rceil}(\tilde{n}-j\delta n)^{\free(X)-1}|S_j|\le\int_{x=0}^t(\tilde{n}-x)^{\free(X)-1}dx\,.\]  
Plugging this into \eqref{eq:p:Sj}, we get
  \begin{align*}
   \sum_{i=1}^t(p_{X,i}+r_{X,i})&=\int_{x=0}^t\frac{|X(\bbA,\bbC)|(\tilde{n}-x)^{\free(X)-1}\free(X)}{\tilde{n}^{\free(X)-1}m|\cJ|}dx\pm\frac{\free(X)\delta n}{|\cJ|}\pm \sum_{j=1}^{\lceil \tfrac{t}{\delta n}\rceil}\tfrac{200\delta n\delta_{j\delta n}}{\gamma|\cJ|}\\
   &=|X(\bbA,\bbC)|\frac{\tilde{n}^{\free(X)}-(\tilde{n}-t)^{\free(X)}}{\tilde{n}^{\free(X)-1}m|\cJ|}\pm\tfrac14\delta_t m\,,
  \end{align*}
  as desired, where the final line follows from the choice of $\delta$, from~\eqref{eq:deltas}, and since $\tfrac{n}{|\cJ|}<2m$.
 \end{claimproof}
 
 We would now like to argue that $\sum_{i=1}^tp_{\iIe,i}$ is a good estimate for the number of edge labels in $\iIe$ used in the labelling process up to time $t$. However, we are not able to do this in one step. Instead, we define a \emph{fine estimate} $q_{\iIe,i}$, which plays the same r\^ole as $p_{\iIe,i}$ except that we condition on the behaviour of Algorithm~\ref{alg:label} up to and including the time $h$ at which we label $v_h=\prt{v_i}$. We will see that at least $\sum_{i=1}^tp_{\iIe,i}$ is a good estimate for $\sum_{i=1}^tq_{\iIe,i}$. We define similarly a fine estimate $q_{X,i}$ corresponding to $p_{X,i}$. We write down formulae valid for $i\ge 2$, when $v_h=\prt{v_i}$ exists. 
   \reminder{$q_{\iIe,i}$}
   \reminder{$q_{X,i}$}
  \begin{equation}\label{eq:defq}\begin{split}
  q_{\iIe,i}&:=\frac{\big|\big\{a\in\adm(\psi_h(v_h),\iJ(v_i);\setVert_h,\setEdge_h):|a-\psi_h(v_h)|\in\iIe\big\}\big|\tilde{n}^2}{|\setVert_h|^2\ell}  \\
  q_{X,i}&:=\hspace{-11mm}\sum_{a\in\adm(\psi_h(v_h),\iJ(v_i);\setVert_h,\setEdge_h)}\hspace{-11mm}\frac{\big|\big\{X'\in X(\setVert_h,\setEdge_h):a\in\chosenv{X'}{X}\text{ or }|a-\psi_h(v_h)|\in\chosene{X'}{X}\big\}\big|}{|\setVert_h|^2\ell\tilde{n}^{-2}}
 \end{split}\end{equation}
 Additionally, we define $q_{\iIe,1}=q_{X,1}=0$.
 
 Observe that $\tfrac{|\setVert_h||\setEdge_h|\ell}{\tilde{n}|\bbC|}$ is, by~\ref{quasi:X} with the structure $\Xthree{\psi_h(v_h),\iJ(v_i)}$, a good estimate for the size of the set of vertex labels from which we will label $v_i$, ignoring any changes that may occur between time $h$ and the time $i$ when we label $v_i$. Thus $q_{\iIe,i}$ is a good estimate for the expectation of labelling $v_i$ in such a way as to use an edge label in $\iIe$, if we ignore any changes that might occur between times $h$ and $i$. We will see that it is usually reasonable to ignore such changes. Similarly, $q_{X,i}$ is a good estimate for the expected number of structures following the pattern $X$ whose chosen labels contain either the vertex or edge label used at time $i$, ignoring any changes between times $h$ and $i$, and estimating the size of the set of vertex labels from which we label $v_i$ by $\tfrac{|\setVert_h||\setEdge_h|\ell}{\tilde{n}|\bbC|}$.
 
 We now show that the partial sums of the crude estimates are, with high probability, good estimates for the partial sums of the fine estimates. There are two parts to this. First, we will argue that if $v_i\prt{v_i}\not\in\mathcal{R}$, then we have $\Exp[q_{\iIe,i}|\psi_{h-1}]\approx p_{\iIe,i}$, and similarly for the $q_{X,i}$. In other words, $p_{\iIe,i}$ is a good estimate for the expectation of $q_{\iIe,i}$ conditioned on the labelling history up to the time immediately before labelling $\prt{v_i}$, whatever that history might be (as long as it maintains quasirandomness). This is the main combinatorial work in our proof. Second, we observe that the effect of the remaining terms where $v_i\prt{v_i}\in\mathcal{R}$ is small, simply because $\mathcal{R}$ is small, and apply Lemma~\ref{lem:seqhoeff} to argue that the sum of conditional expectations is with high probability close to the partial sum $\sum_{i=1}^tq_{\iIe,i}$, and similarly for the $q_{X,i}$.
 
 \begin{claim}\label{cl:q}
  With probability at least $1-4n^{-1}$ the following holds. For each $\iIe\in\cIe$, each $X\in\cX$, each $1\le t\le n$, and each $1\le k\le\big\lceil\tfrac{t}{\delta n}\big\rceil$, if $(\setVert_i,\setEdge_i)$ is $\delta_i$-quasirandom for each $1\le i<t$, we have
  \begin{align}
   \label{eq:sumqedged}\sum_{i=(k-1)\delta n+1}^{\max(t,k\delta n)}q_{\iIe,i}&=\sum_{i=(k-1)\delta n+1}^{\max(t,k\delta n)}p_{\iIe,i}\pm\frac{2000\delta_{k\delta n} \delta m}{\gamma^4}\;\mbox{, and}\\
   \label{eq:sumqX}\sum_{i=(k-1)\delta n+1}^{\max(t,k\delta n)}q_{X,i}&=\sum_{i=(k-1)\delta n+1}^{\max(t,k\delta n)}p_{X,i}\pm\tfrac{10^6\delta m \delta_{k\delta n}}{\gamma^7}\;.
  \end{align}
 \end{claim} 
 \begin{claimproof}[Proof of Claim~\ref{cl:q}]
  For this proof, for each $0\le i\le n$, let $\hist_i$ denote the history up to, but not including, the labelling of $v_{i+1}$. Thus $\hist_0$ is the empty history. 
 
  We begin with~\eqref{eq:sumqedged}. Given $\iIe$ and $1\le t\le n$, and $1\le k\le\big\lceil\tfrac{t}{\delta n}\big\rceil$, we define random variables for each $1\le h<t$ by
\reminder{$Y_h$}
  \[Y_h=\sum_{i:h<i\le t} q_{\iIe,i}\cdot \ind_{v_iv_h\in E(T)}\cdot \ind_{(k-1)\delta n+1\le i\le\max(t,k\delta n)} \,.\]
%  We have $q_{\iIe,1}=0$, and hence
We have
\begin{align}
\sum_{h=1}^{t-1}Y_h&
=\sum_{h=1}^{t-1}\;\sum_{i:h<i\le t} q_{\iIe,i}\cdot \ind_{v_iv_h\in E(T)}\cdot \ind_{(k-1)\delta n+1\le i\le\max(t,k\delta n)}\,.
\label{eq:BSS}
\end{align}  
Let us now fix any $i$ and look at the coefficient of $q_{\iIe,i}$ on the right-hand side of~\eqref{eq:BSS}. Firstly, the coefficient is never more than~$1$ since there is at most one $h$ which makes the indicator $\ind_{v_iv_h\in E(T)}$ non-zero, namely that corresponding to the parent of $i$ (c.f.~\ref{pre:degen}). Secondly, the coefficient is zero outside the range $[(k-1)\delta n+1,\max(t,k\delta n)]$ due to the indicator $\ind_{(k-1)\delta n+1\le i\le\max(t,k\delta n)}$. On the other hand, if these two conditions are fulfilled, then the coefficient of $q_{\iIe,i}$ is~$1$. Indeed, in that case we can always find a parent $v_h$ with $h<i$. (Note that in our setting, $i$ is always more than~$1$, and so $v_i$ always has a parent.) We conclude that
\begin{equation}
\label{eq:sumJ}
\sum_{h=1}^{t-1}Y_h=\sum_{i=(k-1)\delta n+1}^{\max(t,k\delta n)}q_{\iIe,i}\;.
\end{equation}

  We will apply Lemma~\ref{lem:seqhoeff} to estimate the sum on the left-hand side. As in the proof of Claim~\ref{cl:p}, we let $\cE_t$ be the event that $(\setVert_i,\setEdge_i)$ is $\delta_i$-quasirandom for each $1\le i<t$, and the good event of Claim~\ref{cl:sizeAC} holds. In particular, for each $h$ we have $|\setVert_h|,|\setEdge_h|\ge\tfrac{\gamma}{2}\tilde{n}$.
  
  We need to show that, assuming $\cE_t$, we can give good bounds on $\sum_{h=1}^{t-1}\Exp[Y_h|\hist_{h-1}]$. In turn, to obtain such bounds it is enough to show that $\Exp[q_{\iIe,i}|\hist_{h-1}]\approx p_{\iIe,i}$ for each $1<h<i\le t$ with $v_iv_h\in E(T)\setminus\mathcal{R}$. The terms with $h=1$ or $v_iv_h\in\mathcal{R}$ contribute at most $\tfrac{\Pc n}{\log n}+\eps n$ to the sum by assumption on $\Delta(T)$ and by~\ref{pre:sizeR}, which is small enough to ignore.
  
  Suppose we have $1<h<i\le t$, with $v_h=\prt{v_i}$ and $v_iv_h\in E(T)\setminus \mathcal{R}$. We say that $(a,a')$ is an \emph{admissible pair}\reminder{admissible pair} if $a\in\iJ(v_h)\cap \setVert_h$ and $a'\in\iJ(v_i)\cap \setVert_h$, and $\big|\psi_{h-1}(\prt{v_h})-a\big|,|a'-a|\in \setEdge_h$ are distinct. Note that since $\iJ(v_h)=\overline{\iJ(v_i)}$ by~\ref{pre:comp}, $a$ and $a'$ are distinct. It follows that $\psi_{h-1}\cup\{v_h\mapsto a,v_i\mapsto a'\}$ is a graceful labelling of $T[v_1,\dots,v_h,v_i]$.
  
  Note that $$\Exp[q_{\iIe,i}|\hist_{h-1}]=\Exp_{a\sim \textsf{UNIFORM}(\adm(\psi_{h-1}(\prt{v_h}),\iJ(v_h);\setVert_h,\setEdge_h))}\big[\Exp\left[q_{\iIe,i}|\hist_{h-1},\psi_h(v_h)=a\right]\big],$$ because we choose $\psi_h(v_h)$ uniformly. Thus, by definition of $q_{\iIe,i}$, we have
  \begin{align*}
    \tfrac{|\setVert_h|^2\ell}{\tilde{n}^2}\Exp[q_{\iIe,i}|\hist_{h-1}]
   =&\frac{\sum_{(a,a')\text{ admissible}}\ind_{|a-a'|\in\iIe\cap \setEdge_h}}{\big|\adm(\psi_{h-1}(\prt{v_h}),\iJ(v_h);\setVert_h,\setEdge_h)\big|}\\
   \eqByRef{eq:XthreeAdm}&
   \frac{\sum_{(a,a')\text{ admissible}}\ind_{|a-a'|\in\iIe\cap \setEdge_h}}{\big|\Xthree{\psi_{h-1}(\prt{v_h}),\iJ(v_h)}(\setVert_h,\setEdge_h)\big|}
   \\
   =&
   \frac{\sum_{c\in\iIe\cap \setEdge_h}\left|\Xtwo{\psi_{h-1}(\prt{v_h}),\overline{\iJ(v_i)},c,\iJ(v_i)}(\setVert_h,\setEdge_h)\right|}{\big|\Xthree{\psi_{h-1}(\prt{v_h}),\iJ(v_h)}(\setVert_h,\setEdge_h)\big|}
   \,.
  \end{align*}
  Therefore,
  \[\Exp[q_{\iIe,i}|\hist_{h-1}]=\frac{\tilde{n}^2\sum_{c\in\iIe\cap \setEdge_h}|\Xtwo{\psi_{h-1}(\prt{v_h}),\overline{\iJ(v_i)},c,\iJ(v_i)}(\setVert_h,\setEdge_h)|}{|\setVert_h|^2\ell\big|\Xthree{\psi_{h-1}(\prt{v_h}),\iJ(v_h)}(\setVert_h,\setEdge_h)\big|}\,.\]
  Because we assume $\cE_t$ and $h<t$, so $(\setVert_h,\setEdge_h)$ is $\delta_h$-quasirandom and we can use Lemma~\ref{lem:struct} to estimate both the $\Xtwo{none}$-term in the numerator and the $\Xthree{none}$-term in the denumerator.
We obtain
  \begin{align*}
   \Exp[q_{\iIe,i}|\hist_{h-1}]\eqBy{\eqref{XtwoiJ},\eqref{XthreeiJ}}&\frac{\tilde{n}^2\sum_{c\in\iIe\cap \setEdge_h}\Big(\ell^2\el{\iJ(v_i),c}|\setVert_h|^3\tilde{n}^{-3}\pm3\delta_h\ell\Big)}{|\setVert_h|^2\ell\big(\ell|\setVert_h|^2\tilde{n}^{-2}\pm2\delta_h\ell\big)}\\
   =&\frac{\tilde{n}\sum_{c\in\iIe\cap \setEdge_h}\el{\iJ(v_i),c}}{|\setVert_h|}\pm\frac{50\delta_h m}{\ell\gamma^4}\\
   \JUSTIFY{by Fact~\ref{fact:alexandria}\ref{en:Cortez}}=&\frac{\tilde{n}|\iIe\cap \setEdge_h|\el{\iJ(v_i),\min(\iIe)}}{|\setVert_h|}\pm\frac{60\delta_h m}{\ell\gamma^4}\,.
  \end{align*}
  By~\ref{quasi:edged} we have $|\iIe\cap \setEdge_h|=m\tfrac{|\setVert_h|}{\tilde{n}}\pm\delta_h m$, so that
  \begin{equation}
  \label{eq:pnt}
  \Exp[q_{\iIe,i}|\hist_{h-1}]=m\el{\iJ(v_i),\min(\iIe)}\pm \tfrac{70\delta_h m}{\gamma^4\ell}=p_{\iIe,i}\pm\tfrac{70\delta_i m}{\gamma^4\ell}\,,
  \end{equation}
  where we use $\delta_h<\delta_i$ since $h<i$.

  For most values of $i$ we actually have a stronger estimate than~\eqref{eq:pnt}. If $\el{\iJ(v_i),c}=0$ for all $c\in\iIe$, then $p_{\iIe,i}=q_{\iIe,i}=0$ by~\eqref{eq:edged:crude} and~\eqref{eq:defq}. That is, in this situation we have
  \begin{equation}
    \label{eq:pntStronger}
    \Exp[q_{\iIe,i}|\hist_{h-1}]=p_{\iIe,i}\,.
  \end{equation}
  For any given $\iIe\in\cIe$, by Fact~\ref{fact:alexandria}\ref{en:Ocasio} there are at most $2\tfrac{\ell}{m}$ sets $\iJ\in\cJ$ such that $\el{\iJ,c}\neq 0$ for some $c\in\iIe$, so by~\ref{pre:chJ}, and choice of $\delta$, for each $1\le k\le \delta^{-1}$, the number of $i\in\{(k-1)\delta n+1,\dots,k\delta n\}$ such that $\el{\iJ(v_i),c}\neq 0$ for some $c\in\iIe$ is at most $2\tfrac{\ell}{m}\cdot\tfrac{2\delta n}{|\cJ|}$. 
  
  We want to sum up~\eqref{eq:pnt} and~\eqref{eq:pntStronger}. To this end, for $h=2,\ldots,t-1$, set\reminder{$\Vh$}
 \begin{equation}\label{eq:defAh}
  \Vh=\{i\in\bbN\::\: (k-1)\delta n+1\le i\le \max(t,k\delta n), \prt{v_i}=v_h,v_iv_h\not\in\mathcal{R}\}\;.
  \end{equation}
  Observe that $\dbigcup_{h=2}^{t-1}\Vh=\{i\in\mathbb N\::\:
  (k-1)\delta n+1 \le i\le \max(t,k\delta n), \prt{v_i}\neq v_1, v_i\prt{v_i}\not\in\mathcal{R}\}$. Thus, we have
  \begin{equation}
  \label{eq:refertodelta}
  \sum_{h=2}^{t-1}\;\sum_{i\in\Vh}\Exp[q_{\iIe,i}|\hist_{h-1}]
  \;\eqBy{\eqref{eq:pnt},~\eqref{eq:pntStronger}}\;
  \Big(\sum_{\substack{i=(k-1)\delta n+1\\ \prt{v_i}\neq v_1\,,\, v_i\prt{v_i}\not\in\mathcal{R}}}^{\max(t,k\delta n)}p_{\iIe,i}\Big)\pm \frac{4\ell\delta n}{m|\cJ|}\cdot \frac{70\delta_{k\delta n} m}{\ell\gamma^4}\,,
  \end{equation}
  since $i\le k\delta n$ and thus $\delta_i\le\delta_{k\delta n}$. Taking into account the at most $\eps n$ values of $i$ with $v_i\prt{v_i}\in\mathcal{R}$ and at most $\tfrac{\Pc n}{\log n}$ values of $i$ with $\prt{v_i}=v_1$, each of which terms contributes an error of at most $1$, and using~\eqref{eq:IJsize}, we have
  \begin{align*}
   \sum_{h=1}^{t-1}\Exp[Y_h|\hist_{h-1}]&=\sum_{h=1}^{t-1}\;\sum_{i=(k-1)\delta n+1\,,\, \prt{v_i}=v_h}^{\max(t,k\delta n)}\Exp[q_{\iIe,i}|\hist_{h-1}]\\
   &=\Big(\sum_{i=(k-1)\delta n+1}^{\max(t,k\delta n)}p_{\iIe,i}\Big)\pm  \frac{600\delta_{k\delta n} \delta m}{\gamma^4}\pm\eps n\pm\tfrac{\Pc n}{\log n}\\
   &=\Big(\sum_{i=(k-1)\delta n+1}^{\max(t,k\delta n)}p_{\iIe,i}\Big)\pm  \frac{1000\delta_{k\delta n} \delta m}{\gamma^4}\,.
  \end{align*}
  
  We are now in a position to apply Lemma~\ref{lem:seqhoeff}, with the random variables $(Y_h)_{h=1}^{t-1}$ satisfying $0\le Y_h\le\deg_T(v_h)$ for each $h$, and with the event $\cE_t$. By~\eqref{eq:treedegsq} we have $\sum_{h=1}^{t-1}\deg_T(v_h)^2\le\tfrac{2\Pc n^2}{\log n}$, so applying Lemma~\ref{lem:seqhoeff} we conclude that with probability at least $1-\exp\big(-\tfrac{2\delta^2n^2\log n}{2\Pc n^2}\big)>1-n^{-10}$, if $\cE_t$ holds, we have
  \begin{align*}
   \sum_{h=1}^{t-1}Y_h
   &\eqByRef{eq:sumJ}
   \sum_{i=(k-1)\delta n+1}^{\max(t,k\delta n)}q_{\iIe,i}=\Big(\sum_{i=(k-1)\delta n+1}^{\max(t,k\delta n)}p_{\iIe,i}\Big)\pm\frac{1000\delta_{k\delta n} \delta m}{\gamma^4}\pm\delta n\\
   &=\Big(\sum_{i=(k-1)\delta n+1}^{\max(t,k\delta n)}p_{\iIe,i}\Big)\pm\frac{2000\delta_{k\delta n} \delta m}{\gamma^4}\,.
  \end{align*}
  Taking the union bound over $1\le t\le n$ and $\iIe\in\cIe$ and $k$, we conclude that with probability at least $1-n^{-2}$, if $\cE_t$ holds then~\eqref{eq:sumqedged} holds for each $1\le t\le n$, each  $\iIe\in\cIe$, and each $k$. Recall that the good event of Claim~\ref{cl:sizeAC} holds with probability at least $1-n^{-1}$. Thus the following event holds with probability at least $1-2n^{-1}$. For each $1\le t\le n$, if $(\setVert_i,\setEdge_i)$ is $\delta_i$-quasirandom for each $1\le i<t$, then~\eqref{eq:sumqedged} holds for each $\iIe\in\cIe$ and each $k$.
  
  The proof that with high probability~\eqref{eq:sumqX} holds follows the same idea, although the combinatorial manipulations are a little more involved. Now, given $X\in\cX$ and $1\le t\le n$, we define
\reminder{$Y_h$}
  \[Y_h=\sum_{h<i\le t}q_{X,i} \cdot \ind_{v_iv_h\in E(T)} \cdot \ind_{(k-1)\delta n+1\le i\le\max(t,k\delta n)}\,.\]
  and again the critical point is to show that, assuming $\cE_t$, for each $1<h<i\le t$ such that $v_iv_h\in E(T)\setminus\mathcal{R}$, we have $\Exp[q_{X,i}|\hist_{h-1}]\approx p_{X,i}$.
  
  As before, given a choice of $1<h<i\le t$ with $v_h=\prt{v_i}$ such that $v_iv_h\not\in\mathcal{R}$, we say that $(a,a')$ is an \emph{admissible pair}\reminder{admissible pair} if $a\in\iJ(v_h)\cap \setVert_h$ and $a'\in\iJ(v_i)\cap \setVert_h$, and $\big|\psi_{h-1}(\prt{v_h})-a\big|,|a'-a|\in \setEdge_h$ are distinct.
  
  By~\eqref{eq:defq}, and since $\psi_h(v_h)$ is chosen uniformly from $\adm\big(\psi_{h-1}(\prt{v_h}),\iJ(v_h);\setVert_h,\setEdge_h\big)$, we have
  \begin{align}\label{eq:q:bigsum}
   \nonumber&\frac{|\setVert_h|^2\ell\Exp[q_{X,i}|\hist_{h-1}]\cdot\big|\adm\big(\psi_{h-1}(\prt{v_h}),\iJ(v_h);\setVert_h,\setEdge_h\big)\big|}{\tilde{n}^2}\\
   =&\sum_{(a,a')\text{ admissible}}\quad \sum_{X'\in X(\setVert_h,\setEdge_h)}\ind_{a'\in\chosenv{X'}{X}\text{ or }|a-a'|\in\chosene{X'}{X}}\\
   \nonumber=& \sum_{X'\in X(\setVert_h,\setEdge_h)}\quad\sum_{(a,a')\text{ admissible}}\ind_{a'\in\chosenv{X'}{X}\text{ or }|a-a'|\in\chosene{X'}{X}}\\
   \nonumber=&\hspace{-4mm} \sum_{X'\in X(\setVert_h,\setEdge_h)}\quad\sum_{(a,a')\text{ admissible}}\hspace{-5mm}\ind_{a'\in\chosenv{X'}{X}}+\ind_{|a-a'|\in\chosene{X'}{X}}-\ind_{a'\in\chosenv{X'}{X}}\ind_{|a-a'|\in\chosene{X'}{X}}\,,
  \end{align}
  where the last equality holds since vertex labels of $X'$ are by definition pairwise distinct, and edge labels of $X'$ are by definition pairwise distinct.
  
  For a given $X'\in X(\setVert_h,\setEdge_h)$, we have
  \begin{equation}\begin{split}\label{eq:q:Xv}
   \sum_{(a,a')\text{ admissible}}\ind_{a'\in\chosenv{X'}{X}}&=\sum_{a'\in\chosenv{X'}{X}\cap\iJ(v_i)}\Big|\Xfour{\psi_{h-1}(\prt{v_h}),a',\iJ(v_h)}(\setVert_h,\setEdge_h)\Big|\\
   &\eqByRef{XfouriJ}\sum_{a'\in\chosenv{X'}{X}\cap\iJ(v_i)}\big(\tfrac{|\setVert_h|^3\ell}{\tilde{n}^3}\pm 2\delta_h\ell\big)\\
   &=\sum_{I\in\freev{X}}\ind_{I\subset\iJ(v_i)}\big(\tfrac{|\setVert_h|^3\ell}{\tilde{n}^3}\pm 2\delta_h\ell\big)\,,
  \end{split}\end{equation}
  where we can apply Lemma~\ref{lem:struct} since $(\setVert_h,\setEdge_h)$ is by assumption $\delta_h$-quasirandom. Similarly, we have
  \begin{equation}\begin{split}\label{eq:q:Xe}
   \sum_{(a,a')\text{ admissible}}\ind_{|a-a'|\in\chosene{X'}{X}}&=\sum_{c\in\chosene{X'}{X}}\Big|\Xtwo{\psi_{h-1}(\prt{v_h}),\iJ(v_h),c,\iJ(v_i)}(\setVert_h,\setEdge_h)\Big| \\
   &\eqByRef{XtwoiJ}\sum_{c\in\chosene{X'}{X}}\big(\tfrac{|\setVert_h|^3\ell^2\el{\iJ(v_h),c}}{\tilde{n}^3}\pm 3\delta_h\ell\big)\\
   \JUSTIFY{by Fact~\ref{fact:alexandria}\ref{en:Cortez}}&=\sum_{\ed\in\freee(X)}\big(\tfrac{|\setVert_h|^3\ell^2\el{\iJ(v_h),\diff(\ed;X)}}{\tilde{n}^3}\pm 4\delta_h\ell\big)\,,
  \end{split}\end{equation}
  where the final line follows since $\diff(\ed;X)$ is within $m$ of the edge label chosen for $\ed$ in any $X'$ following the pattern $X$. Finally, since $\chosene{X'}{X}$ is a set of size at most $2$, for each $a'\in\chosenv{X'}{X}$ there are at most $4$ choices of $a$ such that $|a-a'|\in\chosene{X'}{X}$. Since $\chosenv{X'}{X}$ is a set of size at most $2$, in total there are at most $8$ pairs $(a,a')$ such that $a'\in\chosenv{X'}{X}$ and $|a-a'|\in\chosene{X'}{X}$. We therefore have
  \begin{equation}\label{eq:q:Xve}
   \sum_{(a,a')\text{ admissible}}\ind_{a'\in\chosenv{X'}{X}}\ind_{|a-a'|\in\chosene{X'}{X}}\le 8\,.
  \end{equation}
  Observe that the final expression in each of~\eqref{eq:q:Xv},~\eqref{eq:q:Xe} and~\eqref{eq:q:Xve} is independent of $X'$. Thus, substituting into~\eqref{eq:q:bigsum}, we have
  \begin{align}
  \begin{split}
   \label{eq:eqStr}
   &\frac{|\setVert_h|^2\ell\Exp[q_{X,i}|\hist_{h-1}]\cdot\big|\adm\big(\psi_{h-1}(\prt{v_h}),\iJ(v_h);\setVert_h,\setEdge_h\big)\big|}{\tilde{n}^2}\\
   =&\big|X(\setVert_h,\setEdge_h)\big|\sum_{I\in\freev{X}}\ind_{I\subset\iJ(v_i)}\big(\tfrac{|\setVert_h|^3\ell}{\tilde{n}^3}\pm 2\delta_h\ell\big)\\
   +&\big|X(\setVert_h,\setEdge_h)\big|\sum_{\ed\in\freee(X)}\big(\tfrac{|\setVert_h|^3\ell^2\el{\iJ(v_h),\diff(\ed;X)}}{\tilde{n}^3}\pm4\delta_h\ell\big)\pm 8m\\
   \eqBy{F\ref{fact:zada}\ref{en:ZadaB}}&\frac{|\setVert_h|^3\ell^2\big|X(\setVert_h,\setEdge_h)\big|}{\tilde{n}^3}\Big(\sum_{I\in\freev(X)}\frac{\ind_{I\subset\iJ(v_i)}}{\ell}+\sum_{\ed\in\freee(X)}\el{\iJ(v_i),\diff(\ed;X)}\Big)\pm 20\delta_h\ell m\\
   \eqBy{\ref{quasi:X}}&\frac{|\setVert_h|^{3+\free(X)}\ell^2\big|X(\bbA,\bbC)\big|}{\tilde{n}^{3+\free(X)}}\Big(\sum_{I\in\freev(X)}\frac{\ind_{I\subset\iJ(v_i)}}{\ell}+\sum_{\ed\in\freee(X)}\el{\iJ(v_i),\diff(\ed;X)}\Big)\pm 40\delta_h\ell m\\
   \eqByRef{eq:struct:crude}&\frac{|\setVert_h|^{3+\free(X)}\ell^2\big|X(\bbA,\bbC)\big|}{\tilde{n}^{3+\free(X)}}\Big(\frac{\tilde{n}^{\free(X)-1}p_{X,i}}{|X(\bbA,\bbC)|(\tilde{n}-i)^{\free(X)-1}} \Big)\pm 40\delta_h\ell m\\
   =&\frac{|\setVert_h|^{3+\free(X)}\ell^2p_{X,i}}{\tilde{n}^{4}(\tilde{n}-h)^{\free(X)-1}} \pm 50\delta_h\ell m\,.
   \end{split}
  \end{align}
  where for the second equality we use the fact that $\el{\iJ(v_i),\diff(\ed;X)}=\el{\iJ(v_h),\diff(\ed;X)}$ since $\iJ(v_i)=\overline{\iJ(v_h)}$ by~\ref{pre:comp}, for the third we use the assumption that $(\setVert_h,\setEdge_h)$ is $\delta_h$-quasirandom, and for the last line we use the fact $i= h\pm\tfrac{\eps n}{\log n}$, which holds by~\ref{pre:int} since $v_iv_h\not\in\mathcal{R}$.
  
  Now, since $(\setVert_h,\setEdge_h)$ is $\delta_h$-quasirandom, by~\eqref{XthreeiJ} and since  $|\setVert_h|\ge\tfrac{\gamma}{2}\tilde{n}$ by~\eqref{eq:sizeAC}, we have
\begin{equation}
\label{eq:eqKolov}
\big|\adm\big(\psi_{h-1}(\prt{v_h}),\iJ(v_h);\setVert_h,\setEdge_h\big)\big|=\big(|\setVert_h|^2/\tilde{n}^2\big)\ell\pm 2\delta_h\ell=\big(1\pm\tfrac{8\delta_h}{\gamma^2}\big)|\setVert_h|^2\tilde{n}^{-2}\ell\,.
\end{equation}
We can rewrite~\eqref{eq:eqStr} as
\begin{align*}
\Exp[q_{X,i}|\hist_{h-1}]
   &=
\frac{   \frac{|\setVert_h|^{1+\free(X)}\ell p_{X,i}}{\tilde{n}^{2}(\tilde{n}-h)^{\free(X)-1}} \pm\frac{50\delta_h m\tilde n^2}{|\setVert_h|^2}}{\big|\adm\big(\psi_{h-1}(\prt{v_h}),\iJ(v_h);\setVert_h,\setEdge_h\big)\big|}\\
\JUSTIFY{substituting  \eqref{eq:eqKolov} and using that $|\setVert_h|\ge\tfrac{\gamma}{2}\tilde{n}$}&=\frac{|\setVert_h|^{\free(X)-1} p_{X,i}}{(\tilde{n}-h)^{\free(X)-1}}\big(1\pm\tfrac{16\delta_h}{\gamma^2}\big)\pm \frac{2000\delta_h m}{\gamma^7\ell}\;.
\end{align*}
  
By~\eqref{eq:sizeAC}, we have $|\setVert_h|=\tilde{n}-h\pm 10\ell$. Furthermore, by~\eqref{eq:struct:crude} we have $p_{X,i}\le 4m/\ell$. Finally, we have $\free(X)\le 3$. We thus get
  \begin{equation}\begin{split}\label{eq:q:structest}
   \Exp[q_{X,i}|\hist_{h-1}]&=\frac{(\tilde{n}-h\pm10\ell)^{\free(X)-1} p_{X,i}}{(\tilde{n}-h)^{\free(X)-1}}\big(1\pm\tfrac{16\delta_h}{\gamma^2}\big)\pm \frac{2000\delta_h m}{\gamma^7\ell}\\
   &=p_{X,i}\big(1\pm \tfrac{16\delta_h}{\gamma^2}\big)\big(1\pm\tfrac{20\ell}{\gamma n}\big)^2\pm \frac{2000\delta_h m}{\gamma^7\ell}\\
   &=p_{X,i}\pm \frac{100\delta_h m}{\gamma^2\ell}\pm \frac{800m}{\gamma n}\pm \frac{2000\delta_h m}{\gamma^7\ell}=p_{X,i}\pm\tfrac{3000\delta_i m}{\gamma^7\ell}\,.
  \end{split}\end{equation}
  
  As before, for most values of $i$ we obtain a stronger estimate. If $\el{\iJ(v_i),c}=0$ for each $c\in\chosene{X'}{X}$ and each $X'\in X(\bbA,\bbC)$, and $I\cap\iJ(v_i)=\emptyset$ for each $I\in\freev(X)$, then $p_{X,i}=q_{X,i}=0$ by~\eqref{eq:struct:crude} and~\eqref{eq:defq}. In this situation, we can write 
    \begin{equation}
    \label{eq:q:structestStronger}
    \Exp[q_{X,i}|\hist_{h-1}]=p_{X,i}\,.
    \end{equation}
  Since $X$ has at most two free edge labels, for each of which an edge label can be chosen in an interval of length $m$, there are at most $4\tfrac{\ell}{m}$ sets $\iJ\in\cJ$ such that $\el{\iJ,c}\neq 0$ for some $c\in\chosene{X'}{X}$ and $X'\in X(\bbA,\bbC)$. Since $X$ has at most two free vertex labels, there are at most $2\tfrac{\ell}{m}$ sets $\iJ\in\cJ$ such that $I\subset\iJ$ for some $I\in\freev(X)$. Putting this together, for all but at most $6\tfrac{\ell}{m}$ sets $\iJ\in\cJ$, if $\iJ(v_i)=\iJ$ we have~\eqref{eq:q:structestStronger}. For each $k\ge 1$, and each interval $i\in\{(k-1)\delta n+1,\dots,k\delta n\}$, by~\ref{pre:chJ}, for all but at most $\tfrac{2\delta n}{|\cJ|}\cdot 6\tfrac{\ell}{m}$ choices of $i$ we have~\eqref{eq:q:structestStronger}. Thus, using the notation from~\eqref{eq:defAh}, after summing up~\eqref{eq:q:structest} and~\eqref{eq:q:structestStronger} we obtain
   \[\sum_{h=2}^{t-1}\;\sum_{i\in\Vh}\Exp[q_{X,i}|\hist_{h-1}]=\Big(\sum_{\substack{i=(k-1)\delta n+1\\ \prt{v_i}\neq v_1, v_i\prt{v_i}\not\in\mathcal{R}}}^{\max(t,k\delta n)}p_{X,i}\Big)\pm \frac{12\ell\delta n}{m|\cJ|}\cdot \frac{3000\delta_{k\delta n} m}{\ell\gamma^7}\,,\]
   where (as in~\eqref{eq:refertodelta}) since $i\le k\delta n$ we have $\delta_i\le\delta_{k\delta n}$. Since $0\le q_{X,i},p_{X,i}\le 6$ for each $i$, taking into account the at most $\eps n$ values of $i$ with $v_i\prt{v_i}\in\mathcal{R}$ and at most $\frac{\Pc n}{\log n}$ values with $\prt{v_i}=v_1$, we have
   \begin{align*}
    \sum_{h=1}^{t-1}\Exp[Y_h|\hist_{h-1}]&=\sum_{h=1}^{t-1}\;\sum_{i=(k-1)\delta n+1\,,\,\prt{v_i}=v_h}^{\max(t,k\delta n)}\Exp[q_{X,i}|\hist_{h-1}]\\
    &\eqByRef{eq:IJsize}\sum_{i=(k-1)\delta n+1}^{\max(t,k\delta n)}p_{X,i}\pm  \tfrac{10^5\delta m \delta_{k\delta n}}{\gamma^7}\pm 6\eps n\pm 6\tfrac{\Pc n}{\log n}\,.
   \end{align*}   
  Finally, since $0\le Y_h\le 6\deg_T(v_h)$, by Lemma~\ref{lem:seqhoeff} and~\eqref{eq:treedegsq}, with probability at least $1-n^{-10}$, if $\cE_t$ holds then we have
  \begin{align*}
   \sum_{i=(k-1)\delta n+1}^{\max(t,k\delta n)}q_{X,i}=\sum_{h=1}^{t-1}Y_h&=\sum_{i=(k-1)\delta n+1}^{\max(t,k\delta n)}p_{X,i}\pm\tfrac{2\cdot 10^5\delta m \delta_{k\delta n}}{\gamma^7}\pm \delta^2 n\\
   &=\sum_{i=(k-1)\delta n+1}^{\max(t,k\delta n)}p_{X,i}\pm\tfrac{10^6\delta m \delta_{k\delta n}}{\gamma^7}\,.
  \end{align*}
  Taking the union bound over the choices of $t$ and $X$ and $k$, we see that with probability at least $1-n^{-8}$, for each $t$ such that $\cE_t$ holds, we have~\eqref{eq:sumqX} for all $X\in\cX$ and $k$. Since the good event of Claim~\ref{cl:sizeAC} holds with probability at least $1-n^{-1}$, we conclude that the statement holds with probability at least $1-4n^{-1}$, as desired.
 \end{claimproof}

 We are now in a position to complete the proof of Theorem~\ref{thm:mainred} by showing that with high probability $(\setVert_t,\setEdge_t)$ is $\delta_t$-quasirandom for each $0\le t\le n-1$. Suppose that the good events of Claims~\ref{cl:sizeAC},~\ref{cl:p} and~\ref{cl:q} hold; this event, which we denote by $\cE$, has probability at least $1-7n^{-1}$. Let $\hist_0$ be the empty history, and $\hist_i$ denote the history up to and including the labelling of $v_i$ for each $1\le i\le n$.
 
The proof that $(\setVert_t,\setEdge_t)$ is $\delta_t$-quasirandom  goes by induction on $t$, with the base case $t=0$ being trivial. Observe that $(\setVert_1,\setEdge_1)=(\bbA,\bbC)$ is by definition $\delta_1$-quasirandom. Now suppose that $(\setVert_i,\setEdge_i)$ are $\delta_i$-quasirandom for each $1\le i\le t$. 
 
 Now suppose $1<i\le t$ is such that $v_i\prt{v_i}\not\in\mathcal{R}$. By~\ref{pre:int}, $v_h=\prt{v_i}$ satisfies $h\ge i-\tfrac{\eps n}{\log n}$, and thus $(\setVert_h,\setEdge_h)$ and $(\setVert_i,\setEdge_i)$ differ by at most $\tfrac{2\eps n}{\log n}$ vertex labels and at most $\tfrac{2\eps n}{\log n}$ edge labels. Thus we have 
 \begin{align}
\label{eq:brabec1} 
\left|\adm(\psi_h(v_h),\iJ(v_i);\setVert_i,\setEdge_i)\right|
=
\left|\adm(\psi_h(v_h),\iJ(v_i);\setVert_h,\setEdge_h)\right|\pm\tfrac{4\eps n}{\log n}\;,
 \end{align}
 and for any $\iIe\in\cIe$,
 \begin{align}
 \begin{split}
\label{eq:brabec2}  &\big|\big\{a\in\adm(\psi_h(v_h),\iJ(v_i);\setVert_i,\setEdge_i):|a-\psi_h(v_h)|\in\iIe\big\}\big|\\
  =&\big|\big\{a\in\adm(\psi_h(v_h),\iJ(v_i);\setVert_h,\setEdge_h):|a-\psi_h(v_h)|\in\iIe\big\}\big|\pm\tfrac{4\eps n}{\log n}\,.
 \end{split}
 \end{align}
 Therefore,
 \begin{equation}\begin{split}\label{eq:estadm}
  \big|\adm\big(\psi_h(v_h),\iJ(v_i);\setVert_i,\setEdge_i\big)\big|&
  \eqByRef{eq:brabec1}
  \big|\adm\big(\psi_h(v_h),\iJ(v_i);\setVert_h,\setEdge_h\big)\big|\pm \tfrac{4\eps n}{\log n}\\
 \JUSTIFY{by~\eqref{eq:XthreeAdm} and~\eqref{XthreeiJ}} &= \ell\big(|\setVert_h|/\tilde{n}\big)^2\pm2\delta_h\ell\pm\tfrac{4\eps n}{\log n}\\
 \JUSTIFY{we have $\delta_h<\delta_i$ since $h<i$}&=\ell|\setVert_h|^2\tilde{n}^{-2}\pm 3\delta_i\ell\,.
 \end{split}\end{equation}
Thus we have
 \begin{equation}\begin{split}\label{eq:continueme}
  &\Prob\big[|\psi_i(v_i)-\psi_i(\prt{v_i})|\in\iIe\big|\hist_{i-1}\big]\\
  &=\frac{\big|\big\{a\in\adm(\psi_h(v_h),\iJ(v_i);\setVert_i,\setEdge_i):|a-\psi_h(v_h)|\in\iIe\big\}\big|}{\big|\adm\big(\psi_h(v_h),\iJ(v_i);\setVert_i,\setEdge_i\big)\big|}\\
\JUSTIFY{by \eqref{eq:estadm} and \eqref{eq:brabec2}} &=\frac{\big|\big\{a\in\adm(\psi_h(v_h),\iJ(v_i);\setVert_h,\setEdge_h):|a-\psi_h(v_h)|\in\iIe\big\}\big|\pm\tfrac{4\eps n}{\log n}}{\ell|\setVert_h|^2\tilde{n}^{-2}\pm 3\delta_i\ell}\;.
 \end{split}
 \end{equation}
 We now need to make some effort to transform the error $\pm2\delta_i\ell$ from the denominator and the term $\tfrac{4\eps n}{\log n}$ from the nominator in a way that will be convenient later. Let us do some preparations first.  Let us write $w_1:=|\{a\in\adm(\psi_h(v_h),\iJ(v_i);\setVert_h,\setEdge_h):|a-\psi_h(v_h)|\in\iIe\}|$, $w_1':=w_1\pm\tfrac{4\eps n}{\log n}$ and $w_2:=\ell|\setVert_h|^2\tilde{n}^{-2}$. Recall that $|\setVert_h|\ge \gamma \tilde n$, and so
 \begin{equation}\label{eq:AHA}
 w_2\ge \gamma^2 \ell\;.
 \end{equation}
 We have
 \begin{align}
 \label{eq:vlak1}
\frac{\delta_i\ell}{w_2} \leByRef{eq:AHA}\frac{\delta_i}{\gamma^2}
 <0.01 \;.
 \end{align}
 We have
 \begin{equation}\begin{split}\label{eq:manipulateerror}
 \frac{w_1'}{w_2\pm 3\delta_i\ell}&=\frac{w_1}{w_2}\cdot\frac{1}{1\pm \frac{3\delta_i\ell}{w_2}}
 \pm
 \frac{\tfrac{4\eps n}{\log n}}{w_2-3\delta_i\ell}\\
\JUSTIFY{using the fact $\frac{1}{1\pm x}=1\pm 2x$ valid for $|x|<0.2$, c.f.~\eqref{eq:vlak1}} &=\frac{w_1}{w_2}\cdot\left(1\pm \tfrac{6\delta_i\ell}{w_2}\right)
\pm
\frac{4\eps n}{0.97 w_2 \log n}\\
\JUSTIFY{using~\eqref{eq:AHA} and \eqref{eq:vlak1}}&=\frac{w_1}{w_2}\cdot\left(1\pm \tfrac{6\delta_i}{\gamma^2}\right)
\pm
\frac{12\eps n}{\gamma^2\ell \log n}\;.
  \end{split}\end{equation}
 We can thus continue~\eqref{eq:continueme} as follows,
  \begin{equation}\begin{split}\label{eq:estiIe}
   &\Prob\big[|\psi_i(v_i)-\psi_i(\prt{v_i})|\in\iIe\big|\hist_{i-1}\big]\\
 &=
 \frac{\big|\big\{a\in\adm(\psi_h(v_h),\iJ(v_i);\setVert_h,\setEdge_h):|a-\psi_h(v_h)|\in\iIe\big\}\big|\pm\tfrac{4\eps n}{\log n}}{\ell|\setVert_h|^2\tilde{n}^{-2} \pm 3\delta_i\ell}\\
  &=\frac{w_1'}{w_2\pm 3\delta_i\ell}  \eqByRef{eq:manipulateerror}
\frac{w_1}{w_2}\cdot\left(1\pm \tfrac{6\delta_i\ell}{w_2}\right)
\pm
\frac{12\eps n}{\gamma^2\ell \log n}
 \\
 \JUSTIFY{by \eqref{eq:defq}}
 &=q_{\iIe,i}\left(1\pm\frac{6\delta_i}{\gamma^2}\right)\pm \frac{12\eps n}{\gamma^2\ell\log n}\,.
 \end{split}\end{equation}
 
 We now argue that a similar equation for $q_{X,i}$ holds. Given $X\in\cX$, we have
 \begin{align*}
  &\hspace{-3mm}\sum_{a\in\adm(\psi_h(v_h),\iJ(v_i);\setVert_i,\setEdge_i)}\hspace{-14mm}\big|\big\{X'\in X(\setVert_i,\setEdge_i):a\in\chosenv{X'}{X}\text{ or }|a-\psi_h(v_h)|\in\chosene{X'}{X}\big\}\big|\\
  =&\hspace{-3mm}\sum_{a\in\adm(\psi_h(v_h),\iJ(v_i);\setVert_h,\setEdge_h)}\hspace{-14mm}\big|\big\{X'\in X(\setVert_h,\setEdge_h):a\in\chosenv{X'}{X}\text{ or }|a-\psi_h(v_h)|\in\chosene{X'}{X}\big\}\big|\pm 24\cdot\tfrac{2\eps n}{\log n}
 \end{align*} 
 and putting this together with~\eqref{eq:estadm} we obtain
 \begin{align*}
  &\Exp\Big[\big|\big\{X'\in X(\setVert_i,\setEdge_i):\psi_i(v_i)\in\chosenv{X'}{X}\text{ or }|\psi_i(v_i)-\psi_h(v_h)|\in\chosene{X'}{X}\big\}\big|\Big|\hist_{i-1}\Big]\\
  =&\frac{\sum\limits_{a\in\adm(\psi_h(v_h),\iJ(v_i);\setVert_i,\setEdge_i)}\hspace{-14mm}\big|\big\{X'\in X(\setVert_i,\setEdge_i):a\in\chosenv{X'}{X}\text{ or }|a-\psi_h(v_h)|\in\chosene{X'}{X}\big\}\big|}{\big|\adm\big(\psi_h(v_h),\iJ(v_i);\setVert_i,\setEdge_i\big)\big|}\\
  =&\frac{\sum\limits_{a\in\adm(\psi_h(v_h),\iJ(v_i);\setVert_h,\setEdge_h)}\hspace{-14mm}\big|\big\{X'\in X(\setVert_h,\setEdge_h):a\in\chosenv{X'}{X}\text{ or }|a-\psi_h(v_h)|\in\chosene{X'}{X}\big\}\big|\pm\tfrac{50\eps n}{\log n}}{\ell|\setVert_h|^2\tilde{n}^{-2}\pm3\delta_i\ell}\;.
  \end{align*}
  Now, dealing with the error terms as in~\eqref{eq:manipulateerror}, we obtain
 \begin{align}\begin{split}
\label{eq:estX}
 &\Exp\Big[\big|\big\{X'\in X(\setVert_i,\setEdge_i):\psi_i(v_i)\in\chosenv{X'}{X}\text{ or }|\psi_i(v_i)-\psi_h(v_h)|\in\chosene{X'}{X}\big\}\big|\Big|\hist_{i-1}\Big]\\ &\eqByRef{eq:defq}q_{X,i}\big(1\pm\tfrac{50\delta_i}{\gamma^2}\big)\pm\tfrac{400\eps n}{\gamma^2\ell\log n}\,.
\end{split}
\end{align}

 We are finally in position to estimate $\sum_{i=2}^t\Prob\big[|\psi_i(v_i)-\psi_i(\prt{v_i})|\in\iIe\big|\hist_{i-1}\big]$, which is a key quantity in order to verify~\ref{quasi:edged}. Putting together Claim~\ref{cl:q} and~\eqref{eq:estiIe}, we have
 \begin{align}
 \begin{split}\label{lastlabel}
  \sum_{i=2}^t\Prob\big[|\psi_i(v_i)-\psi_i(\prt{v_i})|\in\iIe\big|\hist_{i-1}\big]&=\sum_{i=1}^t\Big(q_{\iIe,i}\big(1\pm\tfrac{6\delta_i}{\gamma^2}\big)\pm\tfrac{12\eps n}{\gamma^2\ell\log n}\Big)\\
  &=\sum_{k=1}^{\big\lceil\tfrac{t}{\delta n}\big\rceil}\Big(\big(1\pm\tfrac{6\delta_{k\delta n}}{\gamma^2}\big)\sum_{i=(k-1)\delta n+1}^{\max(t,k\delta n)}q_{\iIe,i}\Big)\pm\tfrac{12\eps n^2}{\gamma^2\ell\log n}\\
\JUSTIFY{by \eqref{eq:sumqedged}}  &=\sum_{k=1}^{\big\lceil\tfrac{t}{\delta n}\big\rceil}\Big(\big(1\pm\tfrac{6\delta_{k\delta n}}{\gamma^2}\big)\sum_{i=(k-1)\delta n+1}^{\max(t,k\delta n)}p_{\iIe,i}\pm \tfrac{4000\delta_{k\delta n}\delta m}{\gamma^4}\Big)\pm 50\delta m\\
\JUSTIFY{by~\eqref{eq:sumpedged}}  &=\sum_{i=1}^tp_{\iIe,i}\pm\sum_{k=1}^{\big\lceil\tfrac{t}{\delta n}\big\rceil}\Big(\tfrac{6\delta_{k\delta n}}{\gamma^2}\cdot\tfrac{2\delta n}{|\cJ|}+ \tfrac{4000\delta_{k\delta n}\delta m}{\gamma^4}\Big)\pm 50\delta m\\
  \JUSTIFY{by~\eqref{eq:IJsize}}&=\sum_{i=1}^tp_{\iIe,i}\pm\sum_{k=1}^{\big\lceil\tfrac{t}{\delta n}\big\rceil}\Big(\tfrac{24\delta_{k\delta n}\delta m}{\gamma^2}+ \tfrac{4000\delta_{k\delta n}\delta m}{\gamma^4}\Big)\pm 50\delta m\\
  \JUSTIFY{by~\eqref{eq:deltas}, recall $\mu^{-1}\gg \gamma^{-4}$ by Setup~\ref{setup}}&=\sum_{i=1}^tp_{\iIe,i}\pm\tfrac14\delta_t m\,.
  \end{split}
 \end{align}
We apply Lemma~\ref{lem:seqhoeff} for random variables $\big(\ind_{|\psi_i(v_i)-\psi_i(\prt{v_i})|\in\iIe}\big)_{i=2}^t$ and the event $\cE$. Let us go through the assumption of Lemma~\ref{lem:seqhoeff}. Histories naturally generate a filtration, as explained in Section~\ref{ssec:ProbablisticForm}. Obviously, our random variables are bounded from above by~1. Last,~\eqref{lastlabel} calculates the expectation needed for the lemma. Lemma~\ref{lem:seqhoeff} tells us that with probability at least $1-n^{-10}$ if $\cE$ occurs we have
 \[\sum_{i=2}^t\ind_{|\psi_i(v_i)-\psi_i(\prt{v_i})|\in\iIe}=\sum_{i=1}^tp_{\iIe,i}\pm\tfrac14\delta_t m\pm\delta n\,,\]
 and, using Claim~\ref{cl:p}, we conclude that
 \[|\setEdge_{t+1}\cap\iIe|=|\iIe|-\sum_{i=2}^t \ind_{|\psi_i(v_i)-\psi_i(\prt{v_i})|\in\iIe}-\sum_{i=1}^tr_{\iIe,i}=|\iIe|-\tfrac{t}{|\cJ|}\pm\tfrac34\delta_t m\eqByRef{eq:sizeAC}\frac{m|\setVert_{t+1}|}{\tilde{n}}\pm\delta_{t+1} m \]
 as required for~\ref{quasi:edged}. Taking the union bound over $\iIe$, with probability at least $1-n^{-9}$, if $\cE$ occurs we have~\ref{quasi:edged} (with parameter $\delta_{t+1}$) for $(\setVert_{t+1},\setEdge_{t+1})$.
 
 By a similar argument, but using~\eqref{eq:estX} in place of~\eqref{eq:estiIe}, with probability at least $1-n^{-5}$, if $\cE$ occurs we have~\ref{quasi:X}  (with parameter $\delta_{t+1}$) for $(\setVert_{t+1},\setEdge_{t+1})$. Putting these together, we see that if $\cE$ occurs then with high probability we fulfilled Definition~\ref{def:quasirandom}. More precisely, the following event has probability at least $1-n^{-4}$. Either $\cE$ does not occur, or $(\setVert_{t+1},\setEdge_{t+1})$ is $\delta_{t+1}$-quasirandom (and therefore Algorithm~\ref{alg:label} does not fail at time $t+1$).
 
 Taking the union bound over $1\le t\le n$, we conclude by induction that with probability at least $1-n^{-3}$ the following event occurs. Either $\cE$ dos not occur, or $(\setVert_{t+1},\setEdge_{t+1})$ is $\delta_{t+1}$-quasirandom for each $1\le t\le n$ (and so Algorithm~\ref{alg:label} does not fail at any time).
 
 Finally, as we noted above, $\cE$ occurs with probability at least $1-7n^{-1}$, so we conclude that Algorithm~\ref{alg:label} succeeds with probability at least $1-8n^{-1}>0$. When it succeeds, the resulting $\psi_n$ is the desired graceful labelling of $T$ with $(1+\gamma)n$ labels.\hfill{$\square$}

%%%%%%%%%%%%%%%%%%%%%%%%%%%%%%%%%%%%%%%%%%%%%%%%%%%%%%%%%%%%%%%%%%%%%%%%%%%%%%%%%%%%%%%%%%%%%%%%

\section{Concluding remarks}
\label{sec:conclusion}
\subsection{Improvements on Theorem~\ref{thm:main}}
It would be desirable to remove the degree constraint of Theorem~\ref{thm:main} and prove that all sufficiently large trees have approximate graceful labellings. Inspection of our proof reveals that the $\log$-factor in our degree bound is required only in order to have polynomially small probabilities in various places, which in turn we need to apply the union bound. If we could somehow do without this, we would otherwise require only $\Delta(T)\le\Pc n$. In our proof, we take the union bound over all times $t$ and intervals $S\subset[t]$, and over all structures $\cX$ (and over a bounded number of other choices). The former can be avoided: it suffices to establish quasirandomness for $t$ a multiple of $\eps n$, and similarly to discretise the choice of $S$ in~\ref{pre:chJ}. The latter cannot so easily be avoided: we need to ensure that (for example) at each time $t$ there are admissible vertices for labelling each $v_t$, which requires $\Xthree{\psi_{t-1}(\prt{v_t}),\iJ(v_t)}(\setVert_t,\setEdge_t)$ to be non-empty. We cannot afford occasional failures here, which is what we would expect if the probability of its being empty were a small constant, rather than polynomially small. Nevertheless, it is possible that with more care, following a strategy similar to that presented here one can handle all trees with maximum degree $\eps n$, and perhaps even all trees.

It would be very interesting to obtain genuine graceful labellings for a large class of trees. Perhaps the approach here can be put together with the Absorbing Method for this purpose. We are currently investigating this possibility.

\subsection{Bipartite graceful labelling}
We believe our result can be extended to the notion introduced as $\alpha$-valuations by Rosa~\cite{Rosa66}, and now commonly referred to as \emph{bipartite graceful labellings}. Namely, we say that a labelling $\psi$ of a bipartite graph is \emph{bipartite graceful} if it is graceful and all vertex labels in one colour class $V_1$ are smaller than in the other $V_2$. However, already Rosa found out that the concept is too restrictive in that the bipartite version of Conjecture~\ref{conj:GTLC} does not hold (for example, a complete ternary tree of depth $2$ and with $13$ vertices does not have a bipartite graceful labelling).
However, it seems likely that small modifications to our proof can be used to show that the trees as in Theorem~\ref{thm:main} have approximate bipartite graceful labellings. That is, 
the labelling $\psi$ in Theorem~\ref{thm:main} can be taken such that $\psi(V_1)\subset \{1,\ldots,|V_1|+\frac{\gamma n}2\}$ and $\psi(V_2)\subset \{|V_1|+\frac{\gamma n}2+1,\ldots,n+\gamma n\}$. 

Let us briefly sketch the required modifications. By a similar reduction as that to Theorem~\ref{thm:mainred}, we can assume that $|V_1|,|V_2|\ge\gamma n$ in addition to any divisibility properties we need. We need to redefine $\iJ$ and `complementary interval' to replace the `midpoint' $\tfrac{1+\gamma}{2}n$ in our proof with $|V_1|+\tfrac{\gamma n}{2}$. We then need to alter our procedure of picking the intervals $\iJ(v_i)$, insisting that $\iJ(v_i)$ is always below the midpoint for $v_i\in V_1$, and always above for $v_i\in V_2$, but otherwise following the same random procedure. We will also need to alter the distributions $\Corv$ and $\Core$ appropriately. Finally, we apply Algorithm~\ref{alg:label}, and claim that it succeeds with high probability; doing so it returns an approximate bipartite graceful labelling.

In order to prove this, it is necessary to change~\ref{pre:chJ}, taking into account the ratio of vertices in $V_1$ and in $V_2$ on the interval $S$. Similarly, it is necessary to change the quasirandomness definition~\ref{quasi:X}, again taking into account the ratio of vertices in $V_1$ and $V_2$ up to time $t$. We expect that these modifications to our proof suffice, but these promise significant extra technical complexity, and we have not checked the details.

In light of this sketch, we do not think Theorem~\ref{thm:main} should be seen as strong evidence in favour of the Graceful Tree Conjecture. The above sketch would be similarly strong evidence in favour of the statement that all trees admit a bipartite graceful labelling: which is false.

\subsection{Harmonious labellings}\label{ssec:harmonious}
The concept of \emph{harmonious labellings} is the same as of graceful labellings, except that the formula $|f(x)-f(y)|$ defining the label induced on the edge $xy$ is replaced by  $(f(x)+f(y)) \mod q$, where $q$ is the number of edges of the graph in question. This concept was introduced by Graham and Sloane~\cite{GrahamSloane80} who also put forward the counterpart to Conjecture~\ref{conj:GTLC}.
\begin{conjecture}[Harmonious Tree Conjecture]
\label{conj:Harm}
For any $n$-vertex tree $T$ there exists an injective
labelling $\psi:V(T) \rightarrow [n]$ such that the values 
\begin{equation}\label{eq:harmmod}
\Big(\psi(u)+\psi(v)\mod (n-1)\Big)_{uv\in E(T)}
\end{equation}
are pairwise distinct.
\end{conjecture}
Actually, Conjecture~\ref{conj:Harm} can be generalized to Abelian groups. In that setting, the conjecture says that given an $n$-vertex tree $T$ and an Abelian group $\Gamma$ of order $n$ there exists an injective labelling $\psi:V(T) \rightarrow \Gamma$ such that the values $\psi(u)+\psi(v)$ on the edges $uv\in E(T)$ are pairwise distinct. Firstly, note that the  original Conjecture~\ref{conj:Harm} corresponds to $\Gamma=\bbZ_n$. Secondly, observe that if the generalized conjecture holds for all trees and all Abelian groups of the order exactly as of the tree, then it holds also for all trees and all Abelian groups of the order which is at least the order of the tree. 

Conjecture~\ref{conj:Harm} is, too, open. The strongest result by far, obtained very recently by Montgomery, Pokrovskiy, and Sudakov~\cite{MPS:EmbeddingRainbow}, is an asymptotic solution of the group-theoretic version of Conjecture~\ref{conj:Harm}.
\begin{theorem}\label{thm:MPSHarmonious}
For every $\gamma>0$ there exists $n_0\in \bbN$ such for every $n>n_0$, every $n$-vertex tree $T$ and every Abelian group $\Gamma$ of order at least $(1+\gamma)n$, there exists a map $\psi:V(T)\rightarrow \Gamma$ such that the values $\psi(u)+\psi(v)$ on the edges $uv\in E(T)$ are pairwise distinct.
\end{theorem}
Theorem~\ref{thm:MPSHarmonious} is a quick consequence of results on containment of rainbow trees, which are the main focus of~\cite{MPS:EmbeddingRainbow}. The most notable feature of Theorem~\ref{thm:MPSHarmonious}, compared to our Theorem~\ref{thm:main} as well as the tree packing results mentioned in Section~\ref{ssec:treepackings}, is that there is no upper bound on the maximum degree of $T$.

Since the current paper appeared before~\cite{MPS:EmbeddingRainbow}, and since the methods used in~\cite{MPS:EmbeddingRainbow} are very different, we would like to comment how the tools we introduced here may be used to obtain a counterpart of Theorem~\ref{thm:main} for harmonious labellings. The bound on the maximum degree of the tree $T$ would stay $\tfrac{\Pc n}{\log n}$. The relaxation compared to Conjecture~\ref{conj:Harm} would amount to $\psi$ mapping to $[\tilde n]$, $\tilde n:=\lceil(1+\gamma)n\rceil$, and modulus in~\eqref{eq:harmmod} being $\tilde n-1$.

In order for our analysis of Algorithm~\ref{alg:label} to work, we need the property that the marginal distributions of vertex labels and edge labels are close to uniform throughout the whole process. We obtained this by our careful choice of the sets $\iJ(v)$ for $v\in V(T)$. Once we have this, the remaining analysis does not essentially require gracefulness.

In order to modify our method to work for harmonious labellings we would choose $\cJ$ as follows. We consider $[\tilde{n}]$ with the natural cyclic order, and let $\cJ$ be the collection of intervals of length $\ell-1$ in this order starting at $1$, $m+1$,\dots, $\tilde{n}-m+1$. We would not need to define `complementary interval'. Then, in Lemma~\ref{lem:setup}, we would simply choose $\iJ(v)$ independently and uniformly at random from $\cJ$ for each $v\in V(T)$. It is now obvious that if each $v$ were labelled uniformly in $\iJ(v)$ then the result is a uniform distribution of vertex labels, and easy to check that the distribution of edge labels is also uniform. We expect that from this point one can simply follow the algorithm and analysis given, making the obvious small changes to obtain a harmonious rather than graceful labelling. However we have not checked the details.

Note that in our approach we cut the tree $T$ into small subtrees by removing the edges $\mathcal{R}$. This is not used only to assign the intervals $\iJ(v)$, but also to guarantee that most vertices are labelled shortly after their parent is labelled. The former property is not required for harmonious labelling, but the latter property is still required for the analysis.

\section{Acknowledgment}
Part of the project was done while CG and JH visited Max Planck Institute for Informatics. We thank Michal Adamaszek for his contribution to this project in its initial stages. We thank anonymous referees for their comments. 
We also thank Felix Joos who suggested a substantial simplification of the proof of Lemma~\ref{lem:seqhoeff}, and suggestions from Lutz Warnke in the same direction.

\medskip

The contents of this publication reflects only the authors' views and not necessarily the views of the European Commission of the European Union.

%%%% BIBLIOGRAPHY %%%%%%%%%%%%%%%%%%%%%%%%%%%%%%%%%%%%%%%%%%%%%%%

\bibliographystyle{amsplain} \bibliography{bibl}

\end{document}